\documentclass{article}

\usepackage{amsmath}
\usepackage{amssymb}
\usepackage{amsthm}
\allowdisplaybreaks

\usepackage{hyperref}

\usepackage{mathtools}

\usepackage{mismath}

\usepackage{bm} 

\usepackage{graphicx}
\graphicspath{{figs/}} 

\usepackage{tikz}
\usetikzlibrary{math}
\usetikzlibrary{calc}

\usepackage{subcaption}

\usepackage{multirow}

\usepackage{siunitx}
\usepackage{makecell}
\setcellgapes{3pt}

\usepackage{algorithm}

\usepackage{algpseudocode}

\usepackage[capitalize]{cleveref}

\usepackage{easyReview}
\setreviewsoff

\newcommand{\Real}{\mathbb{R}}
\newcommand{\Complex}{\mathbb{C}}
\newcommand{\Integer}{\mathbb{Z}}
\newcommand{\SmallHalf}{{\scriptscriptstyle \frac{1}{2}}}

\DeclareMathOperator{\Div}{div}

\DeclareMathOperator{\Cond}{cond}

\DeclareMathOperator{\Diag}{diag}

\DeclarePairedDelimiter{\RoundBrackets}{(}{)}
\DeclarePairedDelimiter{\CurlyBrackets}{\{}{\}}
\DeclarePairedDelimiter{\SquareBrackets}{[}{]}

\DeclarePairedDelimiter\floor{\lfloor}{\rfloor}

\newtheorem{theorem}{Theorem}[section]

\newtheorem{lemma}[theorem]{Lemma}
\newtheorem{definition}[theorem]{Definition}
\newtheorem{proposition}[theorem]{Proposition}

\usepackage[twoside,
paperwidth=210mm,
paperheight=297mm,
textheight=622pt,
textwidth=468pt,
centering,
headheight=50pt,
headsep=12pt,
footskip=18pt,
footnotesep=24pt plus 2pt minus 12pt,
columnsep=2pc]{geometry}

\title{A fast cosine transformation accelerated method for predicting effective thermal conductivity}
\author{Changqing Ye\thanks{Department of Mathematics, The Chinese University of Hong Kong, Shatin, Hong Kong SAR. cqye@math.cuhk.edu.hk},\quad
Shubin Fu\thanks{Eastern Institute for Advanced Study, Ningbo, China}\quad and
Eric T.~Chung\thanks{Department of Mathematics, The Chinese University of Hong Kong, Shatin, Hong Kong SAR}
}

\begin{document}
\maketitle
\begin{abstract}
  Predicting effective thermal conductivity by solving a Partial Differential Equation (PDE) defined on a high-resolution Representative Volume Element (RVE) is a computationally intensive task.
In this paper, we tackle the task by proposing an efficient and implementation-friendly computational method that can fully leverage the computing power offered by hardware accelerators, namely, graphical processing units (GPUs).
We first employ the Two-Point Flux-Approximation scheme to discretize the PDE and then utilize the preconditioned conjugate gradient method to solve the resulting algebraic linear system.
The construction of the preconditioner originates from FFT-based homogenization methods, and an engineered linear programming technique is utilized to determine the homogeneous reference parameters.
The fundamental observation presented in this paper is that the preconditioner system can be effectively solved using multiple Fast Cosine Transformations (FCT) and parallel tridiagonal matrix solvers.
Regarding the fact that default multiple FCTs are unavailable on the CUDA platform, we detail how to derive FCTs from FFTs with nearly optimal memory usage.
Numerical experiments including the stability comparison with standard preconditioners are conducted for 3D RVEs.
Our performance reports indicate that the proposed method can achieve a $5$-fold acceleration on the GPU platform over the pure CPU platform and solve the problems with $512^3$ degrees of freedom and reasonable contrast ratios in less than $30$ seconds.

\end{abstract}

\section{Introduction}
Composites are ubiquitous nowadays.
When discussions are confined to in-silicon experiments, the inherent small-scale heterogeneity in composite materials renders direct numerical simulations generally impractical, despite the capabilities of modern computing power.
To address this challenge, the widely accepted methodology named upscaling is employed, which aims to establish large-scale homogeneous models by leveraging small-scale heterogeneity.
One example of upscaling is Effective Thermal Conductivity (ETC), which characterizes the overall heat transfer behavior of a heterogeneous material.
Understanding and accurately predicting ETC can enable the design and optimization of advanced materials with desired thermal properties.
For example, in industries such as \replace{aerospace, automotive, and electronics}{aviation, car manufacturing, or electronics}, efficient thermal management is crucial for preventing overheating and ensuring the reliability and performance of components and systems \cite{Moore2014,Puneet2015,Gogoi2021}.

There are two \replace{mainstreams}{ways} of predicting ETC: analytical methods and numerical methods.
Analytical methods pursue analytical expressions for \add{the} effective conductivity using a small set of parameters \cite{Pietrak2014}, where the fundamental idea could be credited to Maxwell \cite{Maxwell2010} and subsequent improvements \replace{could}{can} be found in \cite{Rayleigh1892,Landauer1952,Hasselman1987}.
Numerical methods are built upon homogenization theories \cite{Bensoussan2011,Jikov1994}, which involves solving Partial Differential Equations (PDE) defined on Representative Volume Elements (RVE) that capture the microstructures of composites.
Certainly, analytical methods are extremely efficient due to the explicit utilization of closed expressions, while their accuracy may be questionable in certain cases \cite{Pietrak2014}.
Numerical methods, on the contrary, are computationally intensive but provide more reliable results. This paper focuses on exploring numerical methods for predicting ETC.
it is worth mentioning that calculating ETC based on a single RVE may not be sufficient to capture the heat transfer behavior of a composite material, especially when the material exhibits three scales and strong stochasticity.
Recently, a novel micro-meso-macro multi-scale model was developed by Yang and Guan et al.\ in \cite{Yang2019}, which enables the analysis of transfer behavior in heterogeneous materials with multiple random configurations, and more general applications in engineering practice could be found in \cite{Guan2015,Dong2022}.

\add{The} microstructure within an RVE can be highly complicated, making the generation of body-fitted meshes both time-consuming and prone to failure.
Moreover, the primary information of microstructures is commonly provided by modern digital volume imaging techniques \cite{Larson2002,Poulsen2004,Landis2010}, which leads to the loss of precise geometric descriptions of composites.
Consequently, we always consider pixel/voxel-based representations of RVEs, which are equivalent to multidimensional array data structures in implementations.
While there are certain limitations associated with pixel/voxel-based representations \cite{Segurado2018}, the computational advantages they offer, such as efficient memory access, fast algorithms, and easy implementations, are \remove{much} desirable.

One notable example of harnessing pixel/voxel-based representations of RVEs with fast algorithms is a class of computational homogenization methods based on Fast Fourier Transformations (FFT).
In the realm of micromechanics, periodic Boundary Conditions (BC) are preferred compared to other BCs, due to their ability to generate less biased results in random composites  \cite{Terada2000,Kanit2003}.
Pioneered by Moulinec and Suquet \cite{Moulinec1995, Moulinec1998}, FFT-based computational homogenization has garnered significant attention within the community.
Moulinec and Suquet's scheme \cite{Moulinec1998} commences by introducing a homogeneous reference medium and then transforms the force-balance equation into an integral equation known as the Lippmann-Schwinger equation.
By utilizing Fourier transformations, the convolution operator within the integral equation is converted into a trivial entry-wise product.
Then, fixed-point iterations are applied to obtain the final converged stress solution.
Recall that finite element analysis encompasses two distinct phases: the discretization phase, where suitable finite element function spaces are engineered to discretize the variational form, and the solver phase, where the resulting algebraic linear system is solved.
However, in Moulinec and Suquet's scheme, the cutting line between discretization and solver is opaque.
Taking this into account, subsequent advancements in FFT-based homogenization can be categorized into two groups based on the aspects they address.
Regarding discretization, we here list several remarkable works, such as Willot's scheme \cite{Willot2015}, staggered grid \cite{Schneider2016}, finite volume \cite{Wiegmann2006,Dorn2019} and finite element \cite{Schneider2017}.
In terms of solver techniques, the original fixed-point formulation can be further accelerated using sophisticated techniques such as polarization \cite{Eyre1999, Vinogradov2008} and conjugate gradients \cite{Zeman2010}.
It is worth emphasizing that those references mentioned above are not exhaustive, and interested readers can refer to recent comprehensive review papers \cite{Schneider2021,Lucarini2021} for historical developments and state-of-the-art results in this area.

However, for predicting ETC, imposing periodic BCs on a RVE is not a default option.
Instead, a popular choice is the Dirichlet-Neumann mixed BC \cite{Chen2016, Liu2019,Ngo2019,Ngo2019a}, which is also widely adopted in reservoir simulation for porous medium upscaling \cite{Durlofsky1991}.
The reason may be attributed that this type of BCs aligns closely \add{with} the laboratory setting to obtain ETC or permeability from a real RVE sample.
In this context, we shelve the discussion on the effectiveness and impact of different BCs \cite{Yue2007} and focus on the computational challenges posed by 3D RVEs with a large number of Degrees of Freedom (DoF).

Another notable example of directly working with pixel/voxel-based representations of RVEs is from Lattice Boltzmann Methods (LBM).
LBMs were originally proposed for simulating complex fluid dynamics \cite{Higuera1989, McNamara1988, Aidun2010}, and their application for predicting ETC was initiated by Wang et al.~in \cite{Wang2007}.
One of the main advantages of LBMs is their easy implementation, as there is no need for constructing sparse matrices.
LBMs have also been further developed and applied to various composite materials \cite{Wang2007a, Wang2008, Hussain2020}.
However, LBMs are primarily designed for dynamical processes, whereas predicting effective conductivity essentially requires a stable system.
Therefore, in LBM implementation, a cumbersome convergence criterion is employed to ensure that the numerical solutions have reached a stable state.
Additionally, the convergence rate of LBMs is significantly influenced by the values of certain parameters, as shown in \cite{Yang2022} that fully tuned parameter values can lead to a huge acceleration in convergence.

Because FFTs are tightly bound with periodic BCs, the direct implementation of FFT-based homogenization to obtain ETC with the Dirichlet-Neumann mixed BC imposed is infeasible.
Meanwhile, without periodic BCs, the derivation of \add{the} Lippmann-Schwinger equation in \cite{Moulinec1998} is less straightforward.
However, the ingenious utilization of FFTs remains invaluable, as FFTs are highly modularized functions that can be readily optimized by hardware vendors.
The work by Schneider, Merker, and Kabel in \cite{Schneider2017} provides an alternative perspective, suggesting that FFTs can be directly applied to displacement fields discretized by Q1 elements.
Therefore, the methodology of FFT-based homogenization presented in \cite{Schneider2017} is akin to preconditioners for algebraic linear systems \cite{Saad2003}.
Recently, the idea of employing FFTs as preconditioners has been further developed in \cite{Ladecky2023}.
As highlighted in \cite{Yang2022}, LBMs offer an advantage over nodal finite element methods in terms of deriving ETC, as they are capable of resolving heat flux across the boundary.
Hence, it is crucial to carefully select a suitable discretization scheme that naturally facilitates the calculation of ETC.

In this paper, we begin by introducing the Two-Point Flux-Approximation (TPFA) scheme for discretizing the PDE with the Dirichlet-Neumann mixed BC.
TPFA is widely recognized as the default choice for discretizing flow equations in reservoir simulation \cite{Nardean2022}.
With TPFA, the flow flux across element interfaces can be easily reconstructed, aligning well with the requirement for calculating ETC.
It is worth noting that in \cite{Wiegmann2006,Dorn2019}, the explicit jump discretization method bears similarities to TPFA, although the underlying mathematical derivations differ.
Moreover, based on the derivation presented in the paper, it becomes apparent that TPFA can naturally accommodate orthotropic coefficient profiles and the inhomogeneous Dirichlet BC.
Subsequently, the resulting algebraic linear system is solved by the Preconditioned Conjugate Gradient (PCG) method, where the essential part is the construction of preconditioners.
Inspired by FFT-based homogenization, we introduce homogeneous reference parameters to establish the preconditioner system.
To enhance the performance of the PCG on strongly anisotropic models, we design a linear programming algorithm to optimize the values of reference parameters.
The fundamental observation in our paper is that the preconditioner system could be solved efficiently by leveraging multiple Fast Cosine Transformations (FCT) and parallel tridiagonal matrix solvers.
A theoretical by-product of this observation is that we prove the lower and upper bounds of the condition number of the original algebraic linear system.
We detail the implementations, especially on the CUDA platform, where the out-of-the-box multiple FCTs are unavailable.
Finally, we present numerical experiments that include stability comparisons with standard preconditioners, performance acceleration through GPU utilization, and the impact of lower precision floating formats on homogenization.
\add{The principal contribution of this paper lies in the introduction of the discretization scheme and the construction of the fast solver for predicting ETC.}

We notice that a recent work by Grimm-Strele and Kabel explores the application of FFT-homogenization on nonperiodic BCs \cite{GrimmStrele2021}.
The main distinction between our paper and \cite{GrimmStrele2021} lies in our adoption of the preconditioner framework, which may be more accessible to those who are not familiar with the techniques developed within the FFT-homogenization community.
Another benefit offered by the preconditioner framework, compared to \cite{GrimmStrele2021}, is directly invoking sparse matrix routines provided by hardware vendors, such as Intel's MKL and Nvidia's cuSPARSE, which are typically optimized to target specific hardware architectures.
Moreover, our method is more focused on the BCs for calculating ETC, thereby eliminating the reliance on fast sine transformations as seen in \cite{GrimmStrele2021}.
The utilization of TPFA is rooted in the thermal conductivity setting, and the extension of the derivation presented in the paper to elasticity and full-tensor conductivity is highly nontrivial \cite{Ingram2010,Ambartsumyan2020,Ambartsumyan2020a}.
Based on the boundary integral equation in the periodic setting, a robust FFT-based scheme for porous media is established in \cite{To2020}.
Later, this framework is extended to nonperiodic BCs in \cite{To2021}.
However, to apply the scheme from \cite{To2021} for predicting ETC, a fictitious void space needs to be introduced around the RVE, which will introduce redundant DoFs and complicate the implementation.
A recent study conducted by Zhou and Bhattacharya in \cite{Zhou2021} investigates the utilization of GPUs for accelerating computational micromechanics.
Their settings focus on the conventional periodic BCs, and hence the FCT kernel highlighted in the paper is unnecessary in \cite{Zhou2021}.
\add{There has been a surge of interest in solving nonperiodic boundary value problems in FFT-based homogenization, as evidenced by several recent publications \cite{Monchiet2024, Risthaus2024, Morin2024, Gelebart2024}. In this context, our contribution provides a complementary perspective to the existing literature by emphasizing the preconditioner framework.}

The rest of this paper is organized as follows.
In \cref{sec:preliminaries}, we introduce the PDE model and elaborate on the derivation of the TPFA scheme to the problem.
\Cref{sec:methods} constitutes the core of this paper, where we present the PCG method along with the strategy for determining reference parameters, describe the solution algorithm for solving the preconditioner system with a theoretical by-product, and detail the implementation techniques for multiple FCTs.
Numerical experiments conducted on 3D RVEs with DoF up to $512^3$ are reported in \cref{sec:experiments}. Finally, \cref{sec:conclusions} serves as concluding remarks.

\section{Preliminaries}\label{sec:preliminaries}
\subsection{Model problem}
We consider a three-dimensional RVE denoted by $\Omega$ with dimensions $L^x\times L^y \times L^z$. Our assumptions are as follows: (1) no heat is generated within the medium; (2) thermal flux is continuous across interfaces of different material phases; (3) the system has reached a stable state; (4) a linear constitutive law applies at each material point. Consequently, the governing equation is expressed as:

\begin{equation}\label{eq:governing eq}
  -\Div\RoundBrackets*{\mathbb{K}\nabla p}=0\ \text{in} \ \Omega,
\end{equation}
where $\mathbb{K}(\bm{x})$ is a positive definite matrix at $\bm{x}$. In this paper, we make a further assumption that $\mathbb{K}(\bm{x})$ is \replace{element-wisely}{element-wise} orthotropic, i.e.,
\[
  \mathbb{K}(\bm{x})=\Diag\Big(\kappa^x(\bm{x}),\ \kappa^y(\bm{x}),\ \kappa^z(\bm{x})\Big),
\]
which is essential for our discretization scheme. The PDE in \ref{eq:governing eq} can be solved by applying appropriate BCs. Unlike solid mechanics, where periodic BCs are often preferred and may offer additional benefits over other BCs \cite{Kanit2003}, we introduce a Dirichlet-Neumann mixed BC that aligns with the laboratory setting for determining conductivity parameters \cite{Ngo2019}. Specifically, we first assume that
\[
  p\equiv p_\mathup{in}\ \text{on}\ \Gamma_\mathup{in}\ \text{and}\ p\equiv p_\mathup{out}\ \text{on}\ \Gamma_\mathup{out},
\]
where $\Gamma_\mathup{in}$ and $\Gamma_\mathup{out}$ represent a pair of opposing faces, and $p_\mathup{in}$ and $p_\mathup{out}$ are predetermined constants. Additionally, for the \replace{rest}{remaining} faces denoted by $\Gamma_\mathup{N}$, we enforce the homogeneous Neumann BC. \Cref{fig:RVE} demonstrates a RVE and boundary parts $\Gamma_\mathup{in}$, $\Gamma_\mathup{out}$ and $\Gamma_\mathup{N}$. With those BCs furnished, we can solve for a unique $p$, \add{providing that there exist positive constants $\beta'$ and $\beta''$ such that}
\[
  \beta' \leq \kappa^x(\bm{x}),\ \kappa^y(\bm{x}),\ \kappa^z(\bm{x}) \leq \beta''
\]
for a.e.~$\bm{x}$ in $\Omega$.
For simplicity, we always assume that $\Gamma_\mathup{in}$ and $\Gamma_\mathup{out}$ are aligned in the $z$-direction. Consequently, we can determine ETC along this direction as
\begin{equation} \label{eq:effective formula}
  \kappa^z_\mathup{eff} \coloneqq  \frac{L^z\int_{\Gamma_\mathup{out}}\bm{n}\cdot \mathbb{K}\nabla p \di A}{L^xL^y\RoundBrackets*{p_\mathup{in}-p_\mathup{out}}},
\end{equation}
where \alert{$\bm{n}$} is a unit vector pointing outward from the boundary.
\add{Note that we require that $p_\mathup{in} \neq p_\mathup{out}$ here, otherwise the effective conductivity is not well-defined and the solution of $p$ would be a constant function.}
Similarly, $\kappa_\mathup{eff}^x$ and $\kappa_\mathup{eff}^y$ can be determined by altering the direction of the Dirichlet BCs, while there have been several discussions of deriving ETC in a full-matrix form \cite{Yue2007}. However, in this paper, we focus exclusively on $\kappa_\mathup{eff}^z$ to simplify notations.
It is important to note that the flux term on the boundary, $\bm{n} \cdot \mathbb{K}\nabla p$, needs to be carefully handled in the discretization scheme. Furthermore, the relation $\int_{\Gamma_\mathup{out}} \bm{n}\cdot \mathbb{K}\nabla p \di A + \int_{\Gamma_\mathup{in}} \bm{n}\cdot \mathbb{K}\nabla p \di A = 0$ holds true based on the problem setting.

\begin{figure}[!ht]
  \centering
  \def\NameWidth{0.32}
  \begin{tikzpicture}
    \draw (0, 0) rectangle (4, 4);

    \foreach \x in {0,...,10}{
        \draw[stealth-] (0.4 * \x, 4) -- ++(0, 0.3);
        \draw[-stealth] (0.4 * \x, 0) -- ++(0, -0.2);
      }

    \node[above] at (2, 4.4) {$\Gamma_\mathup{in}$};
    \node[below] at (2, -0.3) {$\Gamma_\mathup{out}$};

    \node[left] at (0, 2) {$\Gamma_\mathup{N}$};
    \node[right] at (4, 2) {$\Gamma_\mathup{N}$};

    \draw[gray] (0, 0) -- ++({-2*\NameWidth}, 0);
    \draw[gray] (0, 4) -- ++({-2*\NameWidth}, 0);
    \draw[gray, stealth-] (-\NameWidth, 0) -- ++(0, {2-\NameWidth});
    \draw[gray, stealth-] (-\NameWidth, 4) -- ++(0, {\NameWidth-2});
    \draw[gray] (4, 0) -- ++({2*\NameWidth}, 0);
    \draw[gray] (4, 4) -- ++({2*\NameWidth}, 0);
    \draw[gray, stealth-] ({4+\NameWidth}, 0) -- ++(0, {2-\NameWidth});
    \draw[gray, stealth-] ({4+\NameWidth}, 4) -- ++(0, {\NameWidth-2});

    \fill[color=gray] (1, 1) circle (0.5);
    \fill[color=gray] (1, 2) circle (0.3);
    \fill[color=gray] (2, 3) circle (0.4);
    \fill[color=gray] (2, 2) circle (0.3);
    \fill[color=gray] (3, 1.2) circle (0.4);
    \fill[color=gray] (3.5, 3) circle (0.3);
    \fill[color=gray] (2.6, 2.4) circle (0.2);
    \fill[color=gray] (0.5, 3) circle (0.25);
    \fill[color=gray] (2.2, 0.2) circle (0.18);
    \fill[color=gray] (3.7, 0.3) circle (0.28);
    \fill[color=gray] (2, 1) circle (0.4);
    \fill[color=gray] (1.2, 3.3) circle (0.4);

    \draw[gray,-stealth] (-2, 2) -- ++(0.8, 0) node[above] {$x$};
    \draw[gray,-stealth] (-2, 2) -- ++(0, -0.8) node[left] {$z$};
    \node[black] at (-2, 2) {$\odot$};

  \end{tikzpicture}
  \caption{An illustration of a RVE and boundary parts $\Gamma_\mathup{in}$, $\Gamma_\mathup{out}$ and $\Gamma_\mathup{N}$.} \label{fig:RVE}
\end{figure}
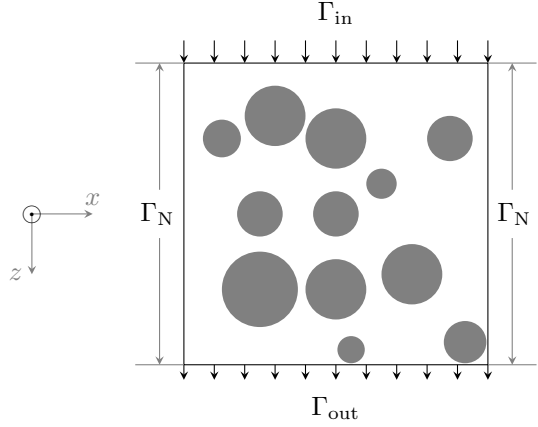

\subsection{Discretization scheme}
We assume that a structured mesh exists on the domain $\Omega$, with a resolution of $N^x \times N^y \times N^z$.
Each element, also known as a voxel, has physical dimensions of $h^x \times h^y \times h^z$, where $h^{\diamond} = L^{\diamond} / N^{\diamond}$ for $\diamond \in {x, y, z}$.
Therefore, the prior knowledge of the RVE could be represented as arrays $\kappa^\diamond_{i,j,k}$ with $\diamond \in \CurlyBrackets*{x, y, z}$ and \alert{$(i, j, k)\in \mathcal{I}_h\coloneqq \CurlyBrackets{(i^*,j^*,k^*)\in \Integer^3:0\leq i^*< N^x,\ 0\leq j^* < N^y,\ 0\leq k^*< N^z}$}.
As a convention, we denote by $\mathcal{T}_h$ the set of all volume elements, $\mathcal{F}_h$ the set of all faces, $\mathcal{F}_h^\circ$ the subset of $\mathcal{F}_h$ by excluding faces on $\Gamma_\mathup{N}$.
The volume element indexed by $(i, j, k) \in \mathcal{I}_h$ can be denoted as $T_{i,j,k} \in \mathcal{T}_h$.
Similarly, any face $F \in \mathcal{F}_h$ can be represented as $F_{i\pm \SmallHalf,j,k}$, $F_{i,j\pm \SmallHalf,k}$, or $F_{i,j,k\pm \SmallHalf}$.
Consequently, we can treat that $\mathbb{K}(\bm{x})$ within the element $T_{i,j,k}$, remains constant and is equal to $\Diag\RoundBrackets*{\kappa^x_{i,j,k},\ \kappa^y_{i,j,k},\ \kappa^z_{i,j,k}}$.

We introduce the flux field $\bm{v}=-\mathbb{K}\nabla p$ to rewrite \cref{eq:governing eq} as a system of equations
\begin{equation}\label{eq:mixed system}
  \left\{
  \begin{aligned}
     & \mathbb{K}^{-1} \bm{v}+\nabla p = 0, \\
     & \Div \bm{v} = 0.
  \end{aligned}
  \right.
\end{equation}
The admissible space for  the flux field is denoted as $H_{0,\Gamma_\mathup{N}}(\Div, \Omega) \coloneqq \{\bm{w}\in L^2(\Omega; \Real^d): \Div \bm{w} \in L^2(\Omega),\  \bm{n}\cdot \bm{w}=0\ \text{on}\ \Gamma_\mathup{N}\}$, while the admissible space for \add{the} pressure is simply $L^2(\Omega)$. By applying test functions to \cref{eq:mixed system}, we obtain the following variational formulation: find $(\bm{v}, p) \in H_{0,\Gamma_\mathup{N}}(\Div, \Omega) \times L^2(\Omega)$, s.t.

\begin{align}
   & \int_{\Omega} \mathbb{K}^{-1} \bm{v} \cdot \bm{w} \di V - \int_\Omega p \Div \bm{w} \di V = -\int_{\Gamma_\mathup{out}} p_\mathup{out} \bm{n} \cdot \bm{w} \di A - \int_{\Gamma_\mathup{in}} p_\mathup{in} \bm{n} \cdot \bm{w} \di A, \nonumber \\
   & \qquad\qquad \forall \bm{w} \in H_{0,\Gamma_\mathup{N}}(\Div, \Omega),                    \label{eq:mixed variational form p1}                                                                                                                  \\
   & -\int_{\Omega} \Div \bm{v} q \di V = 0,\ \forall q \in L^2(\Omega). \label{eq:mixed variational form p2}
\end{align}

%
Note that in \cref{eq:mixed variational form p1}, the Dirichlet boundary terms $p_\mathup{out}$ and $p_\mathup{in}$ appear on the right-hand side, which differs from the commonly used energy minimization formulation.

The calculation of $\kappa_\mathup{eff}^z$ as per \cref{eq:effective formula} involves integrating the flux along the face. Therefore, a direct discretization approach based on the mixed formulation given by \cref{eq:mixed variational form p1,eq:mixed variational form p2} has the potential to offer improved accuracy. Generally, the well-known LBB condition should be satisfied to ensure the stability of finite element spaces of $H_{0,\Gamma_\mathup{N}}(\Div, \Omega) \times L^2(\Omega)$ (ref.~\cite{Boffi2013}). The simplest option is the so-called RT\textsubscript{$0$} spaces, denoted by $\bm{V}_h\times W_h$, where $\bm{V}_h$ is a finite subspace of $H_{0,\Gamma_\mathup{N}}(\Div, \Omega)$ defined \replace{element-wisely}{element-wise} as $(a^xx+b^x, a^yy+b^y, a^zz+b^z)$, and $W_h$ represents the space of \replace{piece-wisely}{piecewise} constant functions on the structural mesh. \Cref{fig:dof} serves as an illustration of bases in $\bm{V}_h$ and $W_h$. To elaborate,  the discretization of $p$ in \cref{eq:mixed variational form p1}, denoted by $p_h$, could simply be represented as $p_{i,j,k}$; the discretization of $\bm{v}$ in \cref{eq:mixed variational form p1,eq:mixed variational form p2}, denoted by $\bm{v}_h$ could be expressed as three 3D arrays $v_{i\pm\SmallHalf,j,k}$, $v_{i,j\pm\SmallHalf,k}$ and $v_{i,j,k\pm\SmallHalf}$. Additionally, the homogeneous BC on $\Gamma_\mathup{N}$ implies that $v_{-\SmallHalf, j, k}=v_{N^x-\SmallHalf, j, k}=v_{i, -\SmallHalf, k}=v_{i, N^y-\SmallHalf,k}=0$ for all $(i,j,k)\in \mathcal{I}_h$. Finally, the value of $\kappa_\mathup{eff}^z$ could be calculated by using the following expression:
\begin{equation}\label{eq:effective formula dis}
  \kappa_\mathup{eff}^z \approx \frac{L^z \sum_{i=0}^{N^x-1}\sum_{j=0}^{N^y-1} v_{i,j,N^z-\SmallHalf}}{N^xN^y(p_\mathup{in}-p_\mathup{out})}.
\end{equation}

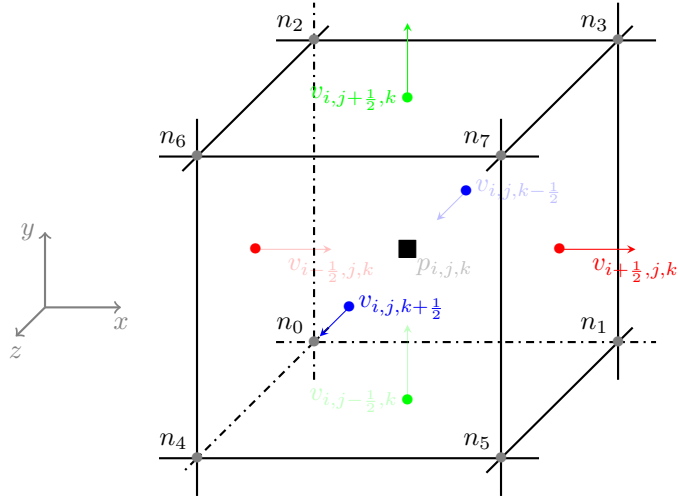
\begin{figure}[!ht]
  \def\OverLen{1.0}
  \def\LENGTH{4.0}
  \def\Opacity{0.25}
  \centering
  \begin{tikzpicture}
    \coordinate (CoodOrigin) at ({-\LENGTH / 2}, {\LENGTH / 2}, \LENGTH);
    \draw[->,gray,thick] (CoodOrigin) -- ++(\OverLen, 0, 0) node[below] {$x$};
    \draw[->,gray,thick] (CoodOrigin) -- ++(0, \OverLen, 0) node[left] {$y$};
    \draw[->,gray,thick] (CoodOrigin) -- ++(0, 0, \OverLen) node[below] {$z$};

    \coordinate (izjzkz) at (0, 0, 0);
    \coordinate (iojzkz) at (\LENGTH, 0, 0);
    \coordinate (izjokz) at (0, \LENGTH, 0);
    \coordinate (iojokz) at (\LENGTH, \LENGTH, 0);
    \coordinate (izjzko) at (0, 0, \LENGTH);
    \coordinate (iojzko) at (\LENGTH, 0, \LENGTH);
    \coordinate (izjoko) at (0, \LENGTH, \LENGTH);
    \coordinate (iojoko) at (\LENGTH, \LENGTH, \LENGTH);


    \draw[thick,dash dot] ({-0.5*\OverLen}, 0, 0) -- ++({\LENGTH+\OverLen}, 0, 0);
    \draw[thick] ({-0.5*\OverLen}, \LENGTH, 0) -- ++({\LENGTH+\OverLen}, 0, 0);
    \draw[thick,dash dot] (0, {-0.5*\OverLen}, 0) -- ++(0, {{\LENGTH+\OverLen}}, 0);
    \draw[thick] (\LENGTH, {-0.5*\OverLen}, 0) -- ++(0, {{\LENGTH+\OverLen}}, 0);
    \draw[thick] ({-0.5*\OverLen}, 0, \LENGTH) -- ++({\LENGTH+\OverLen}, 0, 0);
    \draw[thick] ({-0.5*\OverLen}, \LENGTH, \LENGTH) -- ++({\LENGTH+\OverLen}, 0, 0);
    \draw[thick] (0, {-0.5*\OverLen}, \LENGTH) -- ++(0, {{\LENGTH+\OverLen}}, 0);
    \draw[thick] (\LENGTH, {-0.5*\OverLen}, \LENGTH) -- ++(0, {{\LENGTH+\OverLen}}, 0);
    \draw[thick] (\LENGTH, 0, {-0.5*\OverLen}) -- ++(0, 0, {\LENGTH+\OverLen});
    \draw[thick] (\LENGTH, \LENGTH, {-0.5*\OverLen}) -- ++(0, 0, {\LENGTH+\OverLen});
    \draw[thick] (0, \LENGTH, {-0.5*\OverLen}) -- ++(0, 0, {\LENGTH+\OverLen});
    \draw[thick,dash dot] (0, 0, {-0.5*\OverLen}) -- ++(0, 0, {\LENGTH+\OverLen});

    \node[red] at (0, {0.5*\LENGTH}, {0.5*\LENGTH}) {$\bullet$};
    \draw[red,-stealth, opacity=\Opacity] (0, {0.5*\LENGTH}, {0.5*\LENGTH}) -- ++(\OverLen, 0, 0) node[below]  {$v_{i-\SmallHalf,j,k}$};
    \node[red] at (\LENGTH, {0.5*\LENGTH}, {0.5*\LENGTH}) {$\bullet$};
    \draw[red,-stealth] (\LENGTH, {0.5*\LENGTH}, {0.5*\LENGTH}) -- ++(\OverLen, 0, 0) node[below]  {$v_{i+\SmallHalf,j,k}$};

    \node[green] at ({0.5*\LENGTH}, 0, {0.5*\LENGTH}) {$\bullet$};
    \node[green, left, opacity=\Opacity] at ({0.5*\LENGTH}, 0, {0.5*\LENGTH}) {$v_{i, j-\SmallHalf, k}$};
    \draw[green, -stealth, opacity=\Opacity] ({0.5*\LENGTH}, 0, {0.5*\LENGTH}) -- ++(0, \OverLen, 0);
    \node[green] at ({0.5*\LENGTH}, \LENGTH, {0.5*\LENGTH}) {$\bullet$};
    \node[green,left] at ({0.5*\LENGTH}, \LENGTH, {0.5*\LENGTH}) {$v_{i, j+\SmallHalf, k}$};
    \draw[green,-stealth] ({0.5*\LENGTH}, \LENGTH, {0.5*\LENGTH}) -- ++(0, \OverLen, 0);

    \node[blue] at ({0.5*\LENGTH}, {0.5*\LENGTH}, 0) {$\bullet$};
    \node[blue, right, opacity=\Opacity] at ({0.5*\LENGTH}, {0.5*\LENGTH}, 0) {$v_{i, j, k-\SmallHalf}$};
    \draw[blue, -stealth, opacity=\Opacity5] ({0.5*\LENGTH}, {0.5*\LENGTH}, 0) -- ++(0, 0, \OverLen);
    \node[blue] at ({0.5*\LENGTH}, {0.5*\LENGTH}, \LENGTH) {$\bullet$};
    \node[blue, right] at ({0.5*\LENGTH}, {0.5*\LENGTH}, \LENGTH) {$v_{i, j, k+\SmallHalf}$};
    \draw[blue,-stealth] ({0.5*\LENGTH}, {0.5*\LENGTH}, \LENGTH) -- ++(0, 0, \OverLen);

    \node at ({0.5*\LENGTH}, {0.5*\LENGTH}, {0.5*\LENGTH}) {$\blacksquare$};
    \node[opacity=\Opacity, below right] at ({0.5*\LENGTH}, {0.5*\LENGTH}, {0.5*\LENGTH}) {$p_{i,j,k}$};

    \node[gray] at (izjzkz) {$\bullet$};
    \node[gray] at (iojzkz) {$\bullet$};
    \node[gray] at (izjokz) {$\bullet$};
    \node[gray] at (iojokz) {$\bullet$};
    \node[gray] at (izjzko) {$\bullet$};
    \node[gray] at (iojzko) {$\bullet$};
    \node[gray] at (izjoko) {$\bullet$};
    \node[gray] at (iojoko) {$\bullet$};
    \node[above left] at (izjzkz) {$n_0$};
    \node[above left] at (iojzkz) {$n_1$};
    \node[above left] at (izjokz) {$n_2$};
    \node[above left] at (iojokz) {$n_3$};
    \node[above left] at (izjzko) {$n_4$};
    \node[above left] at (iojzko) {$n_5$};
    \node[above left] at (izjoko) {$n_6$};
    \node[above left] at (iojoko) {$n_7$};

  \end{tikzpicture}
  \caption{An illustration of $\bm{v}_h$ and $W_h$.} \label{fig:dof}
\end{figure}

In the discretization, the term $\int_{\Omega} \mathbb{K}^{-1} \bm{v}_h \cdot \bm{w}_h \di V$ appears as a mass matrix. The trick of TPFA employs a trapezoidal quadrature rule, resulting in a diagonal mass matrix that can be trivially inverted. We first observe that
\[
  \int_{\Omega} \mathbb{K}^{-1} \bm{v}_h \cdot \bm{w}_h \di V =\sum_{(i,j,k)\in \mathcal{I}_h} \int_{T_{i,j,k}}  \RoundBrackets*{\kappa_{i,j,k}^x}^{-1} v_h^x w_h^x + \RoundBrackets*{\kappa_{i,j,k}^y}^{-1} v_h^y w_h^y+\RoundBrackets*{\kappa_{i,j,k}^z}^{-1} v_h^z w_h^z\di V,
\]
and we then introduce an approximation for $\int_{T_{i,j,k}} v_h^x w_h^x \di V$ as
\[
  \int_{T_{i,j,k}} v_h^x w_h^x \di V \approx \frac{h^xh^yh^z}{8}\RoundBrackets*{v_h^x|_{n_0} w_h^x|_{n_0}+\dots+v_h^x|_{n_7} w_h^x|_{n_7}},
\]
where $n_0,\dots,n_7$ are $8$ nodes of the element $T_{i,j,k}$ as shown in \cref{fig:dof}. According to the property of RT\textsubscript{$0$} bases, we can see that $v_h^x|_{n_0}=v_h^x|_{n_2}=v_h^x|_{n_4}=v_h^x|_{n_6}=v_{i-\SmallHalf,j,k}$ and $v_h^x|_{n_1}=v_h^x|_{n_3}=v_h^x|_{n_5}=v_h^x|_{n_7}=v_{i+\SmallHalf,j,k}$. Consequently, we can deduce that
\[
  \int_{T_{i,j,k}} \RoundBrackets*{\kappa_{i,j,k}^x}^{-1} v_h^xw_h^x \di V \approx \frac{h^xh^yh^z}{2}\RoundBrackets*{\kappa_{i,j,k}^x}^{-1} \RoundBrackets*{v_{i-\SmallHalf,j,k}w_{i-\SmallHalf,j,k}+v_{i+\SmallHalf,j,k}w_{i+\SmallHalf,j,k}}.
\]
Summing over the elements, we find that
\[
  \int_{\Omega} \mathbb{K}^{-1} \bm{v}_h \cdot \bm{w}_h \di V \approx h^xh^yh^z \sum_{F\in \mathcal{F}_h^\circ} \kappa_F^{-1} v_Fw_F,
\]
where $\kappa_F$ depends on whether $F$ is an internal face. For instance, if $F$ is an internal face indexed by $F_{i-\SmallHalf,j,k}$, then $\kappa_F=2/\RoundBrackets*{1/\kappa_{i-1,j,k}^x+1/\kappa_{i,j,k}^x}$; If $F$ is not an internal face, i.e., $F \subset \Gamma_\mathup{in} \cup \Gamma_\mathup{out}$, then $\kappa_F = 2\kappa^{z}_{i,j,k}$, where the index $(i,j,k)$ is determined by the element $T_{i,j,k}$ containing $F$.
For $p_h \in W_h$, according to the definition of bases, we can derive that

\begin{align*}
  \int_\Omega p_h \Div \bm{w}_h \di V & =\sum_{(i,j,k)\in \mathcal{I}_h} p_{i,j,k} \int_{T_{i,j,k}} \Div \bm{w}_h \di V=\sum_{(i,j,k)\in \mathcal{I}_h} p_{i,j,k} \int_{\partial T_{i,j,k}} \bm{n} \cdot \bm{w}_h \di A \\
                                      & = h^xh^yh^z \sum_{F\in\mathcal{F}_h^\circ} -p_F^\delta w_F,
\end{align*}

where $p_F^\delta=\RoundBrackets*{p_{i,j,k}-p_{i-1,j,k}}/h^x$ for a face $F$ marked as $F_{i-\SmallHalf,j,k}$, and similarly for other cases. Note that in the definition of $p_F^\delta$, we implicitly take a notation that \alert{$p_{i,j,k}=0$ if $(i,j,k) \in \CurlyBrackets{-1, N^x} \times \CurlyBrackets{-1, N^y} \times \CurlyBrackets{-1, N^z}$}. The discrete right-hand term in \cref{eq:mixed variational form p1} also employs the expression
\[
  -\int_{\Gamma_\mathup{out}} p_\mathup{out} \bm{n} \cdot \bm{w}_h \di A - \int_{\Gamma_\mathup{in}} p_\mathup{in} \bm{n} \cdot \bm{w}_h \di A =-h^xh^y \sum_{F\in \mathcal{F}_h^\circ} \tilde{p}_F w_F,
\]
where
\[
  \tilde{p}_F=\begin{cases}
    p_\mathup{out}, & F \subset \Gamma_\mathup{out}, \\
    -p_\mathup{in}, & F \subset \Gamma_\mathup{in},  \\
    0,              & \text{otherwise}.
  \end{cases}
\]
It is now clear that the discretized variational form for \cref{eq:mixed variational form p1} is equivalent to
\begin{equation}\label{eq:flux}
  v_F= -\kappa_F \RoundBrackets*{p_F^\delta+\frac{\tilde{p}_F}{h^z}}
\end{equation}
by the trapezoidal quadrature rule.
By considering the expression of $v_F$ and rewriting
\[
  -\int_{\Omega} \Div \bm{v}_h q_h \di V = 0 \text{ as } \sum_{F\in \mathcal{F}_h^\circ} v_F q_F^\delta=0,
\]
we can obtain a bilinear form on $W_h$ as $a(p_h, q_h) \coloneqq \sum_{F\in \mathcal{F}_h^\circ} \kappa_F p_F^\delta q_F^\delta$, while $\tilde{p}_F$ in \cref{eq:flux} appears on the right-hand side.
If we partition $\mathcal{F}_h^\circ$ as $\mathcal{F}_h^*\cup \mathcal{F}_h^\mathup{in} \cup \mathcal{F}_h^\mathup{out}$, where we denote the set of all internal faces by $\mathcal{F}_h^*$, the set of all faces on $\Gamma_\mathup{in}$ by $\mathcal{F}_h^\mathup{in}$ and the set of all faces on $\Gamma_\mathup{out}$ by $\mathcal{F}_h^\mathup{out}$. We can also rewrite $a(\cdot, \cdot)$ as
\begin{align*}
  a(p_h, q_h) & = \sum_{F\in \mathcal{F}_h^*} \kappa_F p_F^\delta q_F^\delta + \sum_{F\in \mathcal{F}_h^\mathup{in}} \kappa_F p_F^\delta q_F^\delta + \sum_{F\in \mathcal{F}_h^\mathup{out}} \kappa_F p_F^\delta q_F^\delta                                                      \\
              & = \sum_{F\in \mathcal{F}_h^*} \kappa_F p_F^\delta q_F^\delta + \sum_{i=0}^{N^x-1}\sum_{j=0}^{N^y-1} \frac{2\kappa^z_{i,j,0}}{(h^z)^2} p_{i,j,0}q_{i,j,0} + \sum_{i=0}^{N^x-1}\sum_{j=0}^{N^y-1} \frac{2\kappa^z_{i,j,N^z-1}}{(h^z)^2} p_{i,j,N^z-1}q_{i,j,N^z-1} \\
              & \coloneqq a^*(p_h, q_h) + a^\mathup{in}(p_h, q_h) + a^\mathup{out}(p_h, q_h).
\end{align*}
By applying scaling as $\tilde{\kappa}^\diamond_{i,j,k}=\kappa^\diamond_{i,j,k}/(h^\diamond)^2$ for $\diamond\in \CurlyBrackets{x,y,z}$, we can derive the following expressions,
\begin{equation}\label{eq:scaled bilinear forms}
  \begin{split}
    a^*(p_h, q_h)            & = \sum_{i=1}^{N^x-1}\sum_{j=0}^{N^y-1}\sum_{k=0}^{N^z-1} \tilde{\kappa}_{i-\SmallHalf,j,k} \RoundBrackets*{p_{i,j,k}-p_{i-1,j,k}}\RoundBrackets*{q_{i,j,k}-q_{i-1,j,k}}       \\
    & \qquad +\sum_{i=0}^{N^x-1}\sum_{j=1}^{N^y-1}\sum_{k=0}^{N^z-1}\tilde{\kappa}_{i,j-\SmallHalf,k} \RoundBrackets*{p_{i,j,k}-p_{i,j-1,k}}\RoundBrackets*{q_{i,j,k}-q_{i,j-1,k}}  \\
    & \qquad +\sum_{i=0}^{N^x-1}\sum_{j=0}^{N^y-1}\sum_{k=1}^{N^z-1}\tilde{\kappa}_{i,j,k-\SmallHalf} \RoundBrackets*{p_{i,j,k}-p_{i,j,k-1}}\RoundBrackets*{q_{i,j,k}-q_{i,j,k-1}}, \\
    a^\mathup{in}(p_h, q_h)  & =  \sum_{i=0}^{N^x-1}\sum_{j=0}^{N^y-1} 2\tilde{\kappa}_{i,j,0}^z p_{i,j,0}q_{i,j,0},                                                                                         \\
    a^\mathup{out}(p_h, q_h) & = \sum_{i=0}^{N^x-1}\sum_{j=0}^{N^y-1} 2\tilde{\kappa}_{i,j,N^z-1}^z p_{i,j,N^z-1}q_{i,j,N^z-1},
  \end{split}
\end{equation}
which can be utilized to simplify the implementation.

The complete routine consists of several steps: first solve the algebraic linear system $\mathtt{A}\mathtt{p}=\mathtt{b}$ derived by the aforementioned scheme; then retrieve the flux via the formula \cref{eq:flux}; finally apply \cref{eq:effective formula}  to obtain the homogenized coefficient $\kappa_\mathup{eff}^z$.

\section{Methods}\label{sec:methods}
\subsection{Preconditioned Conjugate Gradient}
Proper preconditioning is essential to ensure efficient performance when solving the algebraic linear system using iterative methods \cite{Saad2003}. Considering the expressions in \cref{eq:scaled bilinear forms} that demonstrate heterogeneity through terms such as $\tilde{\kappa}_{i-\SmallHalf,j,k}$, $\tilde{\kappa}_{i,j-\SmallHalf,k}$, $\tilde{\kappa}_{i,j-\SmallHalf,k}$, $\tilde{\kappa}^z_{i,j,0}$ and $\tilde{\kappa}^z_{i,j,N^z-1}$, we can extend the fundamental idea from FFT-based homogenization methods \cite{Moulinec1995,Moulinec1998,Schneider2017,Zeman2010}. This idea involves introducing a spatially homogeneous reference medium and constructing a preconditioner via three bilinear forms $a^*_\mathup{ref}(\cdot,\cdot)$, $a^\mathup{in}_\mathup{ref}(\cdot,\cdot)$ and $a^\mathup{out}_\mathup{ref}(\cdot,\cdot)$ defined as
\begin{alignat*}{1}
  a^*_\mathup{ref}(z_h, q_h)            & =\kappa_\mathup{ref}^x \sum_{i=1}^{N^x-1}\sum_{j=0}^{N^y-1}\sum_{k=0}^{N^z-1}  \RoundBrackets*{z_{i,j,k}-z_{i-1,j,k}}\RoundBrackets*{q_{i,j,k}-q_{i-1,j,k}}                                                                               \\
                                        & \qquad +\kappa_\mathup{ref}^y \sum_{i=0}^{N^x-1}\sum_{j=1}^{N^y-1}\sum_{k=0}^{N^z-1} \RoundBrackets*{z_{i,j,k}-z_{i,j-1,k}}\RoundBrackets*{q_{i,j,k}-q_{i,j-1,k}}                                                                         \\
                                        & \qquad +\kappa_\mathup{ref}^z\sum_{i=0}^{N^x-1}\sum_{j=0}^{N^y-1}\sum_{k=1}^{N^z-1} \RoundBrackets*{z_{i,j,k}-z_{i,j,k-1}}\RoundBrackets*{q_{i,j,k}-q_{i,j,k-1}}, \stepcounter{equation}\tag{\theequation}\label{eq:refer bilinear forms} \\
  a^\mathup{in}_\mathup{ref}(z_h, q_h)  & =  2\kappa_\mathup{ref}^\mathup{in} \sum_{i=0}^{N^x-1}\sum_{j=0}^{N^y-1}  z_{i,j,0}q_{i,j,0},                                                                                                                                             \\
  a^\mathup{out}_\mathup{ref}(z_h, q_h) & = 2\kappa_\mathup{ref}^\mathup{out}\sum_{i=0}^{N^x-1}\sum_{j=0}^{N^y-1}z_{i,j,N^z-1}q_{i,j,N^z-1}                                                                                                                                         \\
\end{alignat*}
for all $z_h$ and $q_h \in W_h$, where $\kappa_\mathup{ref}^x$, $\kappa_\mathup{ref}^y$, $\kappa_\mathup{ref}^z$, $\kappa_\mathup{ref}^\mathup{in}$ and $\kappa_\mathup{ref}^\mathup{out}$ are adjustable positive constants. Let $\mathtt{A}_\mathup{ref}$ be the matrix derived from the bilinear form $a_\mathup{ref}(\cdot, \cdot)\coloneqq a_\mathup{ref}^*(\cdot,\cdot)+a_\mathup{ref}^\mathup{in}(\cdot,\cdot)+a_\mathup{ref}^\mathup{out}(\cdot,\cdot)$. The PCG method for solving the algebraic system $\mathtt{A}\mathtt{p}=\mathtt{b}$ is stated in \cref{alg:pcg}. Note that the operation $\mathtt{A}_\mathup{ref}^{-1}\mathtt{r}$ in \cref{alg:pcg} is executed \replace{in a matrix-free manner}{independently of the original sparse matrix $\mathtt{A}$}. This implies that the explicit matrix representation of $\mathtt{A}_\mathup{ref}^{-1}$ is unnecessary, and instead, we implement a linear map $\mathtt{r} \mapsto \mathtt{A}_\mathup{ref}^{-1} \mathtt{r}$.

\begin{algorithm}[!ht]
  \caption{The PCG method for solving the linear algebraic system $\mathtt{A}\mathtt{p}=\mathtt{b}$.} \label{alg:pcg}
  \begin{algorithmic}[1]
    \Require The sparse matrix $\mathtt{A}$, the right-hand vector $\mathtt{b}$, the output vector $\mathup{p}$, temporal vectors $\mathtt{r}$, $\mathtt{w}$ and $\mathtt{z}$, temporal scalars $\rho$, $\rho'$, $\alpha$
    \State $\mathtt{r} \leftarrow \mathtt{b} - \mathtt{A}\mathtt{p}$, $\mathtt{z} \leftarrow \mathtt{A}_\mathup{ref}^{-1}\mathtt{r}$, $\mathtt{w}\leftarrow \mathtt{z}$, $\rho \leftarrow \mathtt{r} \cdot \mathtt{z}$
    \While{not meet the exit criterion}
    \State $\mathtt{z} \leftarrow \mathtt{A}\mathtt{w}$
    \State $\alpha \leftarrow \rho / (\mathtt{z} \cdot \mathtt{w})$
    \State $\mathtt{p} \leftarrow \mathup{p} + \alpha \mathtt{w}$
    \State $\mathtt{r} \leftarrow \mathtt{r} - \alpha \mathtt{z}$
    \State $\mathtt{z} \leftarrow \mathtt{A}_\mathup{ref}^{-1}\mathtt{r}$
    \State $\rho'\leftarrow \mathtt{r} \cdot \mathtt{z}$
    \State $\mathtt{w} \leftarrow \mathtt{z} + (\rho'/\rho) \mathtt{w}$
    \State $\rho \leftarrow \rho'$
    \EndWhile
    \State \Return $\mathtt{p}$
  \end{algorithmic}
\end{algorithm}

To complete the method, the reference parameters $\kappa_\mathup{ref}^x$, $\kappa_\mathup{ref}^y$, $\kappa_\mathup{ref}^z$, $\kappa_\mathup{ref}^\mathup{in}$ and $\kappa_\mathup{ref}^\mathup{out}$ still need to be determined, while the convergence theory of PCG can offer valuable insights in this regard.
Taking $\Cond(\mathtt{A}_\mathup{ref}^{-1}\mathtt{A}) \coloneqq \lambda_\mathup{max}(\mathtt{A}_\mathup{ref}^{-1}\mathtt{A})/\lambda_\mathup{min}(\mathtt{A}_\mathup{ref}^{-1}\mathtt{A})$ the condition number of the the preconditioned system, the convergence rate of PCG has an estimate as
\[
  2\RoundBrackets*{\frac{\sqrt{\Cond(\mathtt{A}_\mathup{ref}^{-1}\mathtt{A})}-1}{\sqrt{\Cond(\mathtt{A}_\mathup{ref}^{-1}\mathtt{A})}+1}}^n,
\]
where $n$ is the iteration number \cite{Dolean2015}. The following lemma provides an estimate of $\Cond(\mathtt{A}_\mathup{ref}^{-1}\mathtt{A})$.
\begin{lemma}\label{lem:cond}
  If there exist constants $0< \Lambda'\leq \Lambda''$ such that
  \[
    \Lambda' a_\mathup{ref}(z_h, q_h) \leq a(z_h, q_h) \leq \Lambda'' a_\mathup{ref}(z_h, q_h)
  \]
  for all $z_h$ and $q_h \in W_h$, then $\Cond(\mathtt{A}_\mathup{ref}^{-1}\mathtt{A})\leq \Lambda'' / \Lambda'$.
\end{lemma}

The proof \add{of} \cref{lem:cond} is trivial, and we omit it here. By comparing \cref{eq:scaled bilinear forms} and \cref{eq:refer bilinear forms}, we can obtain the following expressions for $\Lambda'$ and $\Lambda''$:
\begin{align*}
  \Lambda'  & = \min\CurlyBrackets*{\frac{\kappa^x_{\mathup{min}}}{\kappa^x_\mathup{ref}}, \frac{\kappa^y_{\mathup{min}}}{\kappa^y_\mathup{ref}}, \frac{\kappa^z_{\mathup{min}}}{\kappa^z_\mathup{ref}}, \frac{\kappa^\mathup{in}_{\mathup{min}}}{\kappa^\mathup{in}_\mathup{ref}}, \frac{\kappa^\mathup{out}_{\mathup{min}}}{\kappa^\mathup{in}_\mathup{ref}}}, \\
  \Lambda'' & = \max\CurlyBrackets*{\frac{\kappa^x_{\mathup{max}}}{\kappa^x_\mathup{ref}}, \frac{\kappa^y_{\mathup{max}}}{\kappa^y_\mathup{ref}}, \frac{\kappa^z_{\mathup{max}}}{\kappa^z_\mathup{ref}}, \frac{\kappa^\mathup{in}_{\mathup{max}}}{\kappa^\mathup{in}_\mathup{ref}}, \frac{\kappa^\mathup{out}_{\mathup{max}}}{\kappa^\mathup{in}_\mathup{ref}}},
\end{align*}
where
\begin{equation} \label{eq:parameters for linear programming}
  \begin{alignedat}{2}
    & \kappa^x_{\mathup{min}}  \coloneqq  \min_{(i,j,k)}\tilde{\kappa}_{i-\SmallHalf,j,k},\quad    && \kappa^x_{\mathup{max}}  \coloneqq  \max_{(i,j,k)}\tilde{\kappa}_{i-\SmallHalf,j,k},   \\
    & \kappa^y_{\mathup{min}}  \coloneqq  \min_{(i,j,k)}\tilde{\kappa}_{i,j-\SmallHalf,k},\quad    && \kappa^y_{\mathup{max}}  \coloneqq  \max_{(i,j,k)}\tilde{\kappa}_{i,j-\SmallHalf,k},   \\
    & \kappa^z_{\mathup{min}}  \coloneqq  \min_{(i,j,k)}\tilde{\kappa}_{i,j,k-\SmallHalf},\quad    && \kappa^z_{\mathup{max}}  \coloneqq  \max_{(i,j,k)}\tilde{\kappa}_{i,j,k-\SmallHalf},   \\
    & \kappa^\mathup{in}_{\mathup{min}} \coloneqq \min_{(i,j)}\tilde{\kappa}_{i,j,0}^z,\quad       && \kappa^\mathup{in}_{\mathup{max}} \coloneqq \max_{(i,j)}\tilde{\kappa}_{i,j,0}^z,      \\
    & \kappa^\mathup{out}_{\mathup{min}} \coloneqq \min_{(i,j)}\tilde{\kappa}_{i,j,N^z-1}^z,\quad  && \kappa^\mathup{out}_{\mathup{max}} \coloneqq \max_{(i,j)}\tilde{\kappa}_{i,j,N^z-1}^z. \\
  \end{alignedat}
\end{equation}
Therefore, it is reasonable to suggest that minimizing $\Lambda''/\Lambda'$ could yield an appropriate configuration for the reference parameters $\SquareBrackets*{\kappa_\mathup{ref}^x,\ \kappa_\mathup{ref}^y,\ \kappa_\mathup{ref}^z,\ \kappa_\mathup{ref}^\mathup{in},\ \kappa_\mathup{ref}^\mathup{out}}$. It is also interesting to observe that the minimization of $\Lambda''/\Lambda'$ is equivalent to the following linear programming problem:
\begin{equation}\label{eq:linear programming}
  \begin{aligned}
    \min        & \quad \lambda''-\lambda'                                                                                                                             \\
    \text{s.t.} & \quad c^x+\lambda' \leq \log \kappa^x_\mathup{min},\ c^y+\lambda' \leq \log \kappa^y_\mathup{min},\ c^z+\lambda' \leq \log \kappa^z_\mathup{min},    \\
                & \quad c^\mathup{in} + \lambda' \leq \log \kappa^\mathup{in}_\mathup{min},\ c^\mathup{out} + \lambda' \leq \log \kappa^\mathup{out}_\mathup{min},     \\
                & \quad c^x+\lambda'' \geq \log \kappa^x_\mathup{max},\ c^y+\lambda'' \geq \log \kappa^y_\mathup{max},\ c^z+\lambda'' \geq \log \kappa^z_\mathup{max}, \\
                & \quad c^\mathup{in} + \lambda'' \geq \log \kappa^\mathup{in}_\mathup{max},\ c^\mathup{out} + \lambda'' \geq \log \kappa^\mathup{out}_\mathup{max}.
  \end{aligned}
\end{equation}
The solution $\SquareBrackets*{c^x_*,\ c^y_*,\ c^z_*,\ c^\mathup{in}_*,\ c^\mathup{out}_*,\ \lambda'_*,\ \lambda''_*}$ to the optimization problem \cref{eq:linear programming} offers a setting of the reference parameters as
\begin{equation}\label{eq:output optimal ref parameters}
  \SquareBrackets*{\kappa_\mathup{ref}^x,\ \kappa_\mathup{ref}^y,\ \kappa_\mathup{ref}^z,\ \kappa_\mathup{ref}^\mathup{in},\ \kappa_\mathup{ref}^\mathup{out}} = \SquareBrackets*{\exp c^x_*,\ \exp c^y_*,\ \exp c^z_*,\ \exp c^\mathup{in}_*,\ \exp c^\mathup{out}_*}.
\end{equation}
We outline the procedure for determining the reference parameters in \cref{alg:refer parameters}. Note that the linear programming problem \cref{eq:linear programming} involves only $7$ decision variables and $10$ constraints, which could be easily tackled by modern linear programming solvers.

\begin{algorithm}[!ht]
  \caption{The method of determining $\kappa_\mathup{ref}^x$, $\kappa_\mathup{ref}^y$, $\kappa_\mathup{ref}^z$, $\kappa_\mathup{ref}^\mathup{in}$ and $\kappa_\mathup{ref}^\mathup{out}$.}\label{alg:refer parameters}
  \begin{algorithmic}[1]
    \Require The coefficient profiles $\CurlyBrackets*{\tilde{\kappa}^x_{i,j,k},\ \tilde{\kappa}^y_{i,j,k},\ \tilde{\kappa}^z_{i,j,k}}_{(i,j,k)\in \mathcal{I}_h}$
    \State Calculate $\kappa^x_\mathup{min},\dots,\kappa^\mathup{out}_\mathup{max}$ via \cref{eq:parameters for linear programming}
    \State Solve \cref{eq:linear programming} with a linear programming solver and obtain the optimal solution $\SquareBrackets*{c^x_*,\ c^y_*,\ c^z_*,\ c^\mathup{in}_*,\ c^\mathup{out}_*,\ \lambda'_*,\ \lambda''_*}$
    \State Calculate $\SquareBrackets*{\kappa_\mathup{ref}^x,\ \kappa_\mathup{ref}^y,\ \kappa_\mathup{ref}^z,\ \kappa_\mathup{ref}^\mathup{in},\ \kappa_\mathup{ref}^\mathup{out}}$ according to the formula \cref{eq:output optimal ref parameters}
    \State \Return $\SquareBrackets*{\kappa_\mathup{ref}^x,\ \kappa_\mathup{ref}^y,\ \kappa_\mathup{ref}^z,\ \kappa_\mathup{ref}^\mathup{in},\ \kappa_\mathup{ref}^\mathup{out}}$
  \end{algorithmic}
\end{algorithm}

\subsection{Preconditioner system solver}
In this subsection, our objective is detailing the algorithm of solving the preconditioner system \replace{$\mathtt{A}_\mathup{ref}\mathtt{p}=\mathtt{r}$}{$\mathtt{A}_\mathup{ref}\mathtt{z}=\mathtt{r}$}.
According to \cite{Strang1999}, There exist four distinct types of Discrete Cosine Transformations (DCT). Regarding the bilinear forms in \cref{eq:refer bilinear forms}, we should adopt the DCT referred to as DCT-2 in \cite{Strang1999}. The definition of this type of 1D DCT is presented below, where the multiplying constants ($1$ in \cref{eq:forward DCT} and $2/N$ in \cref{eq:backword DCT}) in the forward and backward transformation may vary in other literature sources (e.g., \cite{Makhoul1980}).

\begin{definition}[1D DCT]
  The forward 1D DCT of $u \in \Real^N$ to $\hat{u}$ is defined as
  \begin{equation}\label{eq:forward DCT}
    \hat{u}_{i'} = \sum_{i=0}^{N} u_i \cos\RoundBrackets*{\pi\frac{(2i+1)i'}{2N}}
  \end{equation}
  for $i'=0,\dots, N-1$. The backward 1D DCT of $\hat{u} \in \Real^N$ to $u$ is defined as
  \begin{equation}\label{eq:backword DCT}
    u_{i} = \frac{2}{N}\sum_{i'=0}^{N} \hat{u}_{i'} \alpha_{i'}  \cos\RoundBrackets*{\pi\frac{(2i+1)i'}{2N}}
  \end{equation}
  for $i=0,\dots, N-1$, where $\alpha_{i'}$ is equal to $1/2$ if $i'=0$, and $1$ otherwise.
\end{definition}

The DCT can be derived from the eigenvector expansion. Taking
\[
  v^k\coloneqq \SquareBrackets*{\cos\RoundBrackets*{\pi\frac{k}{2N}}, \cos\RoundBrackets*{\pi\frac{3k}{2N}}, \dots, \cos\RoundBrackets*{\pi\frac{(2N-1)k}{2N}}} \in \Real^N
\]
for $k=0,\dots, N-1$ and
\[
  A=\begin{pmatrix}
    1      & -1     & 0      & \cdots & 0      \\
    -1     & 2      & -1     & ~      & \vdots \\
    0      & -1     & 2      & \ddots & 0      \\
    \vdots & ~      & \ddots & \ddots & -1     \\
    0      & \cdots & 0      & -1     & 1
  \end{pmatrix} \in \Real^{N\times N},
\]
we can validate that $Av^k=\lambda_k(A) v^k$ with $\lambda_k(A)=2(1-\cos \frac{k \pi}{N})$, $v^k\cdot v^k=\frac{N}{2}$ if $k=1,\dots,N-1$, and $v^0\cdot v^0=N$. Moreover, if $u$ and $w \in \Real^N$, it is easy to see that $u\cdot Aw= \sum_{i=1}^{N-1} (u_i-u_{i-1})(w_i-w_{i-1})$. Due to $a_\mathup{ref}^\mathup{in}(\cdot,\cdot)$ and $a_\mathup{ref}^\mathup{out}(\cdot,\cdot)$, we should apply DCTs on $x$- and $y$-direction, but not on $z$-direction, i.e.,
\begin{align*}
  z_{i,j,k} & = \frac{4}{N^xN^y}\sum_{i'=0}^{N^x-1}\sum_{j'=0}^{N^y-1} \hat{z}_{i',j',k} \alpha_{i'}\alpha_{j'} \cos\RoundBrackets*{\frac{\pi (2i+1)i'}{2N^x}}\cos\RoundBrackets*{\frac{\pi (2j+1)j'}{2N^y}}, \\
  q_{i,j,k} & = \frac{4}{N^xN^y}\sum_{i'=0}^{N^x-1}\sum_{j'=0}^{N^y-1} \hat{q}_{i',j',k} \alpha_{i'}\alpha_{j'} \cos\RoundBrackets*{\frac{\pi (2i+1)i'}{2N^x}}\cos\RoundBrackets*{\frac{\pi (2j+1)j'}{2N^y}}.
\end{align*}
Then, thanks to the orthogonality of eigenvectors, we can obtain
\begin{align*}
  a^*_\mathup{ref}(z_h, q_h)            & = \frac{4}{N^xN^y}\Big\{\kappa_\mathup{ref}^x\sum_{k=0}^{N^z-1} \sum_{i'=0}^{N^x-1}\sum_{j'=0}^{N^y-1} 2\RoundBrackets*{1-\cos \frac{i'\pi}{N^x}}\alpha_{i'}\alpha_{j'}\hat{z}_{i',j',k} \hat{q}_{i',j',k}                           \\
                                        & \qquad\qquad+\kappa_\mathup{ref}^y\sum_{k=0}^{N^z-1} \sum_{i'=0}^{N^x-1}\sum_{j'=0}^{N^y-1} 2\RoundBrackets*{1-\cos \frac{j'\pi}{N^y}}\alpha_{i'}\alpha_{j'} \hat{z}_{i',j',k} \hat{q}_{i',j',k}                                     \\
                                        & \qquad\qquad+\kappa^z_\mathup{ref}\sum_{k=1}^{N^z-1} \sum_{i'=0}^{N^x-1}\sum_{j'=0}^{N^y-1}\alpha_{i'}\alpha_{j'}\RoundBrackets*{\hat{z}_{i',j',k}-\hat{z}_{i',j',k-1}}\RoundBrackets*{\hat{q}_{i',j',k}-\hat{q}_{i',j',k-1}}\Big\}, \\
  a^\mathup{in}_\mathup{ref}(z_h, q_h)  & = 2\kappa^\mathup{in}_\mathup{ref}\frac{4}{N^xN^y} \sum_{i'=0}^{N^x-1}\sum_{j'=0}^{N^y-1}\alpha_{i'}\alpha_{j'}\hat{z}_{i',j',0}\hat{q}_{i',j',0},                                                                                   \\
  a^\mathup{out}_\mathup{ref}(z_h, q_h) & = 2\kappa^\mathup{out}_\mathup{ref}\frac{4}{N^xN^y} \sum_{i'=0}^{N^x-1}\sum_{j'=0}^{N^y-1}\alpha_{i'}\alpha_{j'}\hat{z}_{i',j',N^z-1}\hat{q}_{i',j',N^z-1}.
\end{align*}
If we denote the vector enumerating $k$ in $\hat{z}_{i',j',k}$ by $\hat{z}_{i',j',\bullet}$ and in $\hat{q}_{i',j',k}$ by $\hat{q}_{i',j',\bullet}$, and define a group of tridiagonal matrices $\CurlyBrackets{T_{i',j'}} \subset \Real^{N^z\times N^z}$ as
\begin{align*}
  T_{i',j'} & = 2\RoundBrackets*{1-\cos \frac{i'\pi}{N^x}} \kappa_\mathup{ref}^x \mathbb{I} +2\RoundBrackets*{1-\cos \frac{j'\pi}{N^y}}\kappa_\mathup{ref}^y \mathbb{I} \\
            & \quad +
  \begin{pmatrix}
    \kappa_\mathup{ref}^z + 2\kappa_\mathup{ref}^\mathup{in} & -\kappa_\mathup{ref}^z & 0                      & \cdots                 & 0                                                         \\
    -\kappa_\mathup{ref}^z                                   & 2\kappa_\mathup{ref}^z & -\kappa_\mathup{ref}^z & ~                      & \vdots                                                    \\
    0                                                        & -\kappa_\mathup{ref}^z & 2\kappa_\mathup{ref}^z & \ddots                 & 0                                                         \\
    \vdots                                                   & ~                      & \ddots                 & \ddots                 & -\kappa_\mathup{ref}^z                                    \\
    0                                                        & \cdots                 & 0                      & -\kappa_\mathup{ref}^z & \kappa_\mathup{ref}^z + 2\kappa_\mathup{ref}^\mathup{out}
  \end{pmatrix},
\end{align*}
where $\mathbb{I}$ is the identity matrix, we can derive a simplified expression for $a_\mathup{ref}(\cdot, \cdot)$ as
\begin{equation} \label{eq:ref bilinear}
  a_\mathup{ref}(z_h, q_h)=\frac{4}{N^xN^y} \sum_{i'=0}^{N^x-1} \sum_{j'=0}^{N^y-1} \alpha_{i'}\alpha_{j'} T_{i',j'}\hat{z}_{i',j',\bullet} \cdot \hat{q}_{i',j',\bullet}.
\end{equation}
The solving process for the algebraic linear systems \replace{$\mathtt{A}_\mathup{ref}\mathtt{p}=\mathtt{r}$}{$\mathtt{A}_\mathup{ref}\mathtt{z}=\mathtt{r}$} is now clear, and we summarize it into \cref{alg:solve ref}, where DCTs have already been replaced by FCT.


\begin{algorithm}[!ht]
  \caption{The method for solving the preconditioner system.}\label{alg:solve ref}
  \begin{algorithmic}[1]
    \Require The output vector \replace{$\mathtt{p}$}{$\mathtt{z}$}, the right-hand vector $\mathtt{r}$, the data for $T_{i',j'}$ with $0\leq i' < N^x$ and $0 \leq j' < N^y$
    \State For each $k$, apply forward FCT transformations to $r_{i,j,k}$ w.r.t.~$x$- and $y$-direction, and obtain $\hat{r}_{i',j',k}$
    \State For each $i'$ and $j'$, solve the system $T_{i',j'}\hat{z}_{i',j',\bullet}=\hat{r}_{i',j',\bullet}$
    \State  For each $k$, apply backward FCT transformations to $\hat{z}_{i',j',k}$ w.r.t.~$x$- and $y$-direction, and obtain $z_{i,j,k}$
    \State \Return \replace{$\mathtt{p}$}{$\mathtt{z}$}
  \end{algorithmic}
\end{algorithm}

An intermediate result from \cref{eq:ref bilinear} is the following proposition, which states the necessity of preconditioning the linear system. \add{Note that the classic theory only provides an upper bound of the condition number \cite{Brenner2008}.}
\begin{proposition}\label{prop:cond}
  Let $K'$ and $K''$ be positive constants such that $K' \leq \tilde{\kappa}_{i,j,k}^x,\tilde{\kappa}_{i,j,k}^y,\tilde{\kappa}_{i,j,k}^z\leq K''$ for all $(i,j,k)\in \mathcal{I}_h$. Then, the following estimate of the condition number of $\mathtt{A}$ holds,
  \[
    C_1\frac{K'}{K''} \RoundBrackets*{N^z}^2 \leq \Cond(\mathtt{A}) \leq C_2\frac{K''}{K'} \RoundBrackets*{N^z}^2,
  \]
  where $C_1$ and $C_2$ are generic positive constant.
\end{proposition}
\begin{proof}
  In this proof, we denote $c$ with a subscript a generic positive constant, and utilize the relation $a \approx b$ to represent $c_1 a \leq b \leq c_2 b$.
  If we  find $\mathtt{B}'$ and $\mathtt{B}''$ such that $\mathtt{B}' \lesssim \mathtt{A} \lesssim \mathtt{B}''$, it becomes evident that
  \[
    \frac{\lambda_\mathup{max}(\mathtt{B}')}{\lambda_\mathup{min}(\mathtt{B}'')} \leq \Cond(\mathtt{A}) \leq \frac{\lambda_\mathup{max}(\mathtt{B}'')}{\lambda_\mathup{min}(\mathtt{B}')}.
  \]
  By recalling $K'\leq \tilde{\kappa}_{i-\SmallHalf,j,k}, \tilde{\kappa}_{i,j-\SmallHalf,k}, \tilde{\kappa}_{i,j,k-\SmallHalf} \leq K''$ due to the property of harmonic averages, we can establish $K'\mathtt{B} \lesssim \mathtt{A} \lesssim K''\mathtt{B}$, where $\mathtt{B}$ is the preconditioner matrix by setting all reference parameters as $1$. The remaining task is to estimate $\lambda_\mathup{max}(\mathtt{B})$ and $\lambda_\mathup{min}(\mathtt{B})$.

  By examining \cref{eq:refer bilinear forms} and utilizing the trivial inequality $(a-b)^2\leq 2(a^2+b^2)$, we can readily deduce that $ \mathtt{B}\mathtt{p}\cdot \mathtt{p} \leq c_3 \abs{\mathtt{p}}^2$ holds for any $\mathtt{p}$. Moreover, by choosing a special $\mathtt{p}^*$ that has a value of $1$ on one entry and $0$ on all other entries, we obtain $\mathtt{B}\mathtt{p}^* \cdot \mathtt{p}^* \leq c_4 \abs{\mathtt{p}^*}^2$. Therefore, we can bound $\lambda_\mathup{max}(\mathtt{B})$ from above and below, i.e., $ \lambda_\mathup{max}(\mathtt{B}) \approx 1$.

  The formula of \cref{eq:ref bilinear} implies $\mathtt{B}$ is a block diagonal matrix under DCTs. Let $\hat{\mathtt{p}}$ be the vector determined by $\hat{p}_{i',j',k}$, then the relation $\abs{\mathtt{p}}^2 \approx \abs{\hat{\mathtt{p}}}^2/\RoundBrackets*{N^xN^y}$ holds, which implies that $\lambda_\mathup{min}(\mathtt{B}) \approx \min_{(i',j')}\lambda_\mathup{min}(T_{i',j'})$. In this case, the tridiagonal matrix $T_{i',j'}$ could be expressed as
  \[
    T_{i',j'}=2(1-\cos\frac{i' \pi}{N^x})\mathbb{I} + 2(1-\cos\frac{j' \pi}{N^y})\mathbb{I}
    +\underbrace{\begin{pmatrix}
        2      & -1     & 0      & \cdots & 0      \\
        -1     & 2      & -1     & ~      & \vdots \\
        0      & -1     & 2      & \ddots & 0      \\
        \vdots & ~      & \ddots & \ddots & -1     \\
        0      & \cdots & 0      & -1     & 2
      \end{pmatrix}}_{\coloneqq B}
    +\underbrace{\Diag(1,\ 0,\ \dots\ ,\ 0,\ 1)}_{\coloneqq C}.
  \]
  The eigenvalues and eigenvectors of the matrix $B$ are well-known, given by
  \[
    \lambda_k(B)=2-2\cos\frac{(k+1)\pi}{N^z+1}
  \]
  and
  \[
    v^k=\SquareBrackets*{\sin \RoundBrackets*{\pi\frac{(k+1)}{N^z+1}}, \sin \RoundBrackets*{\pi\frac{2(k+1)}{N^z+1}},\dots, \sin \RoundBrackets*{\pi\frac{N^z(k+1)}{N^z+1}}}\in \Real^{N^z}.
  \]
  Note that the matrix $C$ is positive semi-definite. Then, we have
  \begin{align*}
    \lambda_\mathup{min}(T_{i',j'}) & \geq \lambda_\mathup{min}\RoundBrackets*{2(1-\cos\frac{i' \pi}{N^x})\mathbb{I} + 2(1-\cos\frac{j' \pi}{N^y})\mathbb{I}+B}      \\
                                    & =                                     2\RoundBrackets*{3-\cos\frac{i' \pi}{N^x}-\cos\frac{j' \pi}{N^y}-\cos\frac{\pi}{N^z+1}},
  \end{align*}
  which gives that $\lambda_\mathup{min}(\mathtt{B}) \geq c_5/\RoundBrackets*{N^z}^2$. We can observe that
  \[
    Cv^0 \cdot v^0=\sin^2 \RoundBrackets*{\frac{\pi}{N^z+1}} + \sin^2\RoundBrackets*{\frac{N^z\pi}{N^z+1}} \leq c_6 \frac{\abs{v^0}^2}{(N^z)^3},
  \]
  which leads that $\lambda_\mathup{min}\RoundBrackets{T_{0,0}} \leq T_{0,0}v^0\cdot v^0/\abs{v^0}^2 \leq c_7 /\RoundBrackets*{N^z}^2$. Consequently, we can conclude that $\lambda_\mathup{min}(\mathtt{B}) \approx 1/ \RoundBrackets*{N^z}^2$.
\end{proof}

\subsection{Remarks on implementations}\label{subsec:remarks}
The efficient implementation of \cref{alg:solve ref} relies on two components, the tridiagonal matrix solver and FCTs. Solving a tridiagonal linear system has a computational complexity of $\mathcal{O}(n)$, where $n$ represents the number of unknowns \cite{Golub2013}. Therefore, the overall cost of solving all tridiagonal systems \cref{alg:solve ref} is at most $\mathcal{O}(N^xN^yN^z)$. LAPACK provides a variety of tridiagonal matrix factorization and solving routines \cite{Anderson1999}, and there are highly optimized packages available depending on specific systems. For instance, Intel's MKL, now contained in oneAPI, is optimized for x86 CPUs.
It is important to note that while the complexity of solving tridiagonal systems is $\mathcal{O}(n)$, this does not necessarily imply satisfactory parallel scalability.
Fortunately, those tridiagonal systems are independent \replace{with}{of} each other, allowing us to naturally leverage a lazy-parallel approach (e.g., with openMP) to maximize efficiency.
On the CUDA platform, a routine named \texttt{gtsv2StridedBatch} from cuSPARSE can meet the requirement. Similarly, multiple 2D forward and backward FCTs in \cref{alg:solve ref} can also be accelerated through parallelism due to the looping of $k$. In FFTW3 \cite{Frigo2005} and cuFFT, the advanced feature \texttt{plan\_many} handles stride array layouts and avoids direct iteration over $k$, while the internal mechanism may not be transparent to users.  However, on most of the platforms such as MKL and CUDA, the default FCTs with the \texttt{plan\_many} feature are not available.

Fortunately, FCTs can be derived from FFTs. In this paper, 2D forward and backward Discrete Fourier Transformations (DFT) take the following forms
\begin{align*}
  \hat{w}_{i',j'} & = \sum_{i=0}^{N^x-1} \sum_{j=0}^{N^y-1} w_{i,j} \exp\RoundBrackets*{-2\i\pi\RoundBrackets{\frac{ii'}{N^x}+\frac{jj'}{N^y}}},                       \\
  w_{i,j}         & = \frac{1}{N^xN^y}\sum_{i'=0}^{N^x-1} \sum_{j'=0}^{N^y-1} \hat{w}_{i,j} \exp\RoundBrackets*{2\i\pi\RoundBrackets{\frac{ii'}{N^x}+\frac{jj'}{N^y}}}
\end{align*}
respectively.
Makhoul proposed three-phase algorithms for 2D forward and backward FCTs, denoted by \texttt{FCT} and \texttt{iFCT} respectively, in \cite{Makhoul1980}. Specifically, \texttt{FCT} involves the following steps,
\begin{equation} \label{eq:FCT}
  \texttt{FCT}:\qquad  \underset{\Real}{v_{i,j}} \xrightarrow{\texttt{\quad FCT-Pre \quad}} \underset{\Real}{w_{i,j}} \xrightarrow{\texttt{\quad \texttt{FFT} \quad}} \underset{\Complex}{\hat{w}_{i',j'}} \xrightarrow{\texttt{\quad \texttt{FCT-Post \quad}}} \underset{\Real}{\hat{v}_{i',j'}},
\end{equation}
where the data types ($\Real$ or $\Complex$) are indicated, \texttt{FCT-Pre} is defined as
\[
  w_{i,j}=\left\{
  \begin{alignedat}{3}
    & v_{2i,2j},\quad                && i=0,\dots, \floor*{\frac{N^x-1}{2}}, \quad    && j=0,\dots, \floor*{\frac{N^y-1}{2}},    \\
    & v_{2N^x-2i-1, 2j},\quad        && i=\floor*{\frac{N^x+1}{2}},\dots,N^x-1, \quad && j=0,\dots, \floor*{\frac{N^y-1}{2}},    \\
    & v_{2i, 2N^y-2j-1},\quad        && i=0,\dots,\floor*{\frac{N^x-1}{2}}, \quad     && j=\floor*{\frac{N^y+1}{2}},\dots,N^y-1, \\
    & v_{2N^x-2i-1, 2N^y-2j-1},\quad && i=\floor*{\frac{N^x+1}{2}},\dots,N^x-1, \quad && j=\floor*{\frac{N^y+1}{2}},\dots,N^y-1,
  \end{alignedat}
  \right.
\]
\texttt{FCT-Post} is defined as
\[
  \hat{v}_{i',j'}=\left\{
  \begin{aligned}
     & \frac{1}{2}\Re\CurlyBrackets*{\exp\RoundBrackets*{-\i\pi\frac{i'}{2N^x}} \RoundBrackets*{\exp\RoundBrackets*{-\i\pi\frac{j'}{2N^y}}\hat{w}_{i',j'}+\exp\RoundBrackets*{\i\pi\frac{j'}{2N^y}}\hat{w}_{i',N^y-j'}}}, \\
     & \qquad\qquad j'=1,\dots,N^y-1,                                                                                                                                                                                     \\
     & \Re\CurlyBrackets*{\exp\RoundBrackets*{-\i\pi\frac{i'}{2N^x}}\hat{w}_{i',0}},\quad j'=0,
  \end{aligned}
  \right.
\]
and \texttt{FFT} is FFT in the forward direction. Similarly, \texttt{iFCT} consists of the following steps,
\begin{equation}\label{eq:iFCT}
  \texttt{iFCT}:\qquad  \underset{\Real}{\hat{v}_{i',j'}} \xrightarrow{\texttt{\quad iFCT-Pre \quad}} \underset{\Complex}{\hat{w}_{i',j'}} \xrightarrow{\texttt{\quad \texttt{iFFT} \quad}} \underset{\Real}{w_{i,j}} \xrightarrow{\texttt{\quad \texttt{iFCT-Post \quad}}} \underset{\Real}{v_{i,j}},
\end{equation}
where \texttt{iFCT-Pre} is defined as
\[
  \hat{w}_{i',j'} =\left\{
  \begin{aligned}
     & \exp\RoundBrackets*{\i\pi\frac{i'}{2N^x}}\exp\RoundBrackets*{\i\pi\frac{j'}{2N^y}}\RoundBrackets*{\hat{v}_{i',j'}-\hat{v}_{N^x-i',N^y-j'}-\i\RoundBrackets*{\hat{v}_{N^x-i',j'}+\hat{v}_{i',N^y-j'}}}, \\
     & \qquad\qquad i'\neq 0,\quad j'\neq 0,                                                                                                                                                                  \\
     & \exp\RoundBrackets*{\i\pi\frac{j'}{2N^y}}\RoundBrackets*{\hat{v}_{0,j'}-\i\hat{w}_{0,N^y-j'}},\quad i'=0,\quad j'\neq 0,                                                                               \\
     & \exp\RoundBrackets*{\i\pi\frac{i'}{2N^x}}\RoundBrackets*{\hat{v}_{i',0}-\i\hat{w}_{N^x-i',0}},\quad i'\neq 0, \quad j'=0,                                                                              \\
     & \hat{v}_{0,0},\quad i'=0, \quad j'=0,
  \end{aligned}
  \right.
\]
\texttt{iFCT-Post} is defined as
\[
  v_{i,j}=\left\{
  \begin{alignedat}{3}
    & w_{\frac{i}{2}, \frac{j}{2}},\quad             && i=0,2,\dots, 2\floor*{\frac{N^x-1}{2}},\quad && j=0,2,\dots,2\floor*{\frac{N^y-1}{2}},  \\
    & w_{\frac{i}{2}, N^y-\frac{j+1}{2}},\quad       && i=0,2,\dots,2\floor*{\frac{N^x-1}{2}},\quad  && j=1,3,\dots,2\floor*{\frac{N^y}{2}}-1,  \\
    & w_{N^x-\frac{i+1}{2}, \frac{j}{2}},\quad       && i=1,3,\dots,2\floor*{\frac{N^x}{2}}-1,\quad  && j=0,2,\dots,2\floor*{\frac{N^y-1}{2}},  \\
    & w_{N^x-\frac{i+1}{2}, N^y-\frac{j+1}{2}},\quad && i=1,3,\dots,2\floor*{\frac{N^x}{2}}-1,\quad  && j=1,3,\dots, 2\floor*{\frac{N^y}{2}}-1.
  \end{alignedat}
  \right.
\]
and \texttt{iFFT} is the backward FFT. The complex array $\hat{w}$ in \cref{eq:FCT,eq:iFCT} satisfies the conjugate symmetry
\begin{equation}\label{eq:conjugate sym}
  \hat{w}_{i',j'}=\overline{\hat{w}_{N^x-i',N^y-j'}},\quad j'=\floor*{\frac{N^y}{2}}+1,\dots, N^y-1,
\end{equation}
which has been employed in FFT packages to reduce memory usage by approximately half. Finally, we can efficiently implement FCTs in \cref{alg:solve ref} using the formulas \cref{eq:FCT,eq:iFCT} and specialized FFTs facilitated by the \texttt{plan\_many} functionality. Note that the memory usage by \cref{eq:FCT,eq:iFCT} is nearly optimal when considering the conjugate symmetry in \cref{eq:conjugate sym}.

\section{Numerical experiments}\label{sec:experiments}
The proposed method relies on four essential modules: (1) BLAS for fundamental vector operations, (2) SpMV for efficient handling of sparse matrix and vector multiplications, (3) multiple 2D FCTs specialized for stride array layouts, and (4) tridiagonal system solvers within LAPACK. Hardware vendors such as Intel, AMD, and Nvidia typically provide highly optimized packages that encompass these required functionalities, except for FCTs. Leveraging these packages directly can significantly enhance overall efficiency. We conducted most of our numerical experiments on a system equipped with dual Intel Xeon Gold 6230 CPUs and 1TB of memory, with GPU acceleration enabled using a Quadro RTX 8000 with 48GB of memory. The proposed method is implemented on the CUDA platform, utilizing cuBLAS, cuSPARSE, and cuFFT libraries to provide the aforementioned modules. However, as mentioned in \cref{subsec:remarks}, we still manually craft CUDA device codes to handle specific operations such as \texttt{FCT-Pre}, \texttt{FCT-Post}, \texttt{iFCT-Pre}, and \texttt{iFCT-Post}. Having familiarity with the CUDA programming model can be advantageous for an efficient implementation. For comparison purposes, we also perform some experiments on a system comprising dual Intel Xeon Gold 6140 CPUs and 192GB of memory, but without GPUs. The implementation on this system utilized MKL for BLAS, SpMV, and LAPACK, while multiple FCTs were performed using FFTW3 that is compiled with the optimization flags ``-O3'' and ``--enable-avx512 --enable-openmp''. The source code for our implementation is hosted on GitHub\footnote{\href{https://github.com/Laphet/fct-homo}{https://github.com/Laphet/fct-homo}}.

In the subsequent numerical experiments, we set domain $\Omega$ as a unit cube $(0,1)^3$ and the exit criterion in \cref{alg:pcg} as
\begin{equation}\label{eq:exit}
  \frac{\text{the $l^2$-norm of $\mathtt{r}$}}{\text{the $l^2$-norm of $\mathtt{b}$}} \leq \mathtt{rtol}.
\end{equation}
Throughout the reports, we denote the number of PCG iterations by $\mathtt{iter}$ and the total DoF by $\mathtt{dof}$. The linear programming solver for determining the reference parameters in \cref{alg:refer parameters} is from SciPy \cite{Virtanen2020}. The execution time measurements are obtained by averaging the wall time records of $10$ rounds of experiments under the same configuration.

\subsection{Investigations on the TPFA scheme}
In this subsection, we aim to examine the convergence of the TPFA scheme w.r.t.~$\mathtt{dof}$, especially in the homogenization setting.

We begin by examining a smooth case where the coefficient matrix is defined as
\[
  \mathbb{K}(x,y,z)=\Diag(\cos(\pi y)+2,\ 2 \exp(z),\ 3\cos(\pi x)+4).
\]
We choose a suitable source term to ensure that
\[
  p(x,y,z)=\cos(\pi x)\cos(\pi y)\exp(z)
\]
satisfies the equation $-\Div(\mathbb{K}\nabla p)=f$. The numerical results corresponding to different $\mathtt{dof}$ and $\mathtt{rtol}$ are presented in \cref{tab:smooth}, where $L^2$\!-error represents the difference between numerical and exact solutions in $L^2(\Omega)$-norm, while the integration over $\Omega$ here is obtained from the midpoint quadrature rule for simplicity. From \cref{tab:smooth}, we can observe that under different $\mathtt{rtol}$, $\mathtt{iter}$ remains almost unchanged w.r.t.~$\mathtt{dof}$. This fact reveals that the dependency of $\mathtt{dof}$ for $\Cond(\mathtt{A})$, as demonstrated by \cref{prop:cond}, is eliminated by the proposed preconditioner. From the last row that $\mathtt{rtol}=\num{1.0e-9}$, we can see that the $L^2$\!-error decreases by a factor of $4$, which indicates a $h^2$ convergence rate in $L^2(\Omega)$-norm. Therefore, the TPFA scheme achieves the same level of accuracy as finite element methods for smooth problems. However, this smooth setting does not align with the homogenization background, which primarily focuses on heterogeneous media with discontinuous or \replace{piecewisely}{piecewise} constant coefficient profiles.

\begin{table}[!ht]
  \footnotesize
  \caption{Numerical results of a smooth model, \add{where $L^2$\!-error represents the difference between numerical and exact solutions in $L^2(\Omega)$-norm}.} \label{tab:smooth}
  \centering
  \makegapedcells
  \begin{tabular}{c c c c c c c c c c c}
    \hline
    \multirow{2}{*}{$\mathtt{rtol}$} & \multicolumn{2}{c}{$\mathtt{dof}=32^3$} & \multicolumn{2}{c}{$\mathtt{dof}=64^3$} & \multicolumn{2}{c}{$\mathtt{dof}=128^3$} & \multicolumn{2}{c}{$\mathtt{dof}=256^3$} & \multicolumn{2}{c}{$\mathtt{dof}=512^3$}                                                                                       \\
    \cline{2-11}
                                     & $L^2$\!-error                           & $\mathtt{iter}$                         & $L^2$\!-error                            & $\mathtt{iter}$                          & $L^2$\!-error                            & $\mathtt{iter}$ & $L^2$\!-error & $\mathtt{iter}$ & $L^2$\!-error & $\mathtt{iter}$ \\
    \hline
    \num{1.0e-5}                     & \num{3.85e-4}                           & \num{14}                                & \num{9.54e-5}                            & \num{14}                                 & \num{2.41e-5}                            & \num{14}        & \num{1.27e-5} & \num{14}        & \num{1.38e-5} & \num{14}        \\
    \num{1.0e-6}                     & \num{3.85e-4}                           & \num{17}                                & \num{9.63e-5}                            & \num{17}                                 & \num{2.44e-5}                            & \num{17}        & \num{6.69e-6} & \num{17}        & \num{2.65e-6} & \num{17}        \\
    \num{1.0e-7}                     & \num{3.85e-4}                           & \num{20}                                & \num{9.61e-5}                            & \num{20}                                 & \num{2.40e-5}                            & \num{20}        & \num{5.94e-6} & \num{20}        & \num{1.42e-6} & \num{20}        \\
    \num{1.0e-8}                     & \num{3.85e-4}                           & \num{22}                                & \num{9.61e-5}                            & \num{23}                                 & \num{2.40e-5}                            & \num{23}        & \num{6.01e-6} & \num{23}        & \num{1.51e-6} & \num{23}        \\
    \num{1.0e-9}                     & \num{3.85e-4}                           & \num{25}                                & \num{9.61e-5}                            & \num{26}                                 & \num{2.40e-5}                            & \num{26}        & \num{6.61e-6} & \num{26}        & \num{1.50e-6} & \num{26}        \\
    \hline
  \end{tabular}
\end{table}

We then consider a heterogeneous configuration, taking
\[
  \mathbb{K}(x,y,z)=\begin{cases}
    \kappa^\mathup{inc}\mathbb{I}, & \sqrt{(x-\frac{1}{2})^2+(y-\frac{1}{2})^2+(z-\frac{1}{2})^2}\leq \frac{1}{4}, \\
    \mathbb{I},                    & \text{otherwise}.
  \end{cases}
\]
Intuitively, this configuration represents a ball embedded within the RVE. The discretized coefficient profiles $\kappa^x_{i,j,k}$, $\kappa^y_{i,j,k}$ and $\kappa^z_{i,j,k}$ are determined based on whether the element center lies within the ball. Since $\kappa^z_\mathup{eff}$ calculated by \cref{eq:effective formula dis} is of utmost importance in practice, we illustrate the convergence of $\kappa^z_\mathup{eff}$ w.r.t.~$\mathtt{dof}$ under different $\kappa^\mathup{inc}$ in \cref{fig:eff convergence dof}, where $\mathtt{rtol}$ is fixed at $\num{1.0e-9}$. Analyzing \cref{fig:eff convergence dof},  we can observe convergence rate of $\kappa^z_\mathup{eff}$ is nearly linear to $h$. A mathematical explanation for this convergence rate is beyond the scope of this paper; however, several theoretical studies exist for other discretization schemes and problem settings \cite{Ye2023,Schneider2022,Schneider2023}. It is also interesting to note that the convergence behavior differs for $\kappa^\mathup{inc} < 1$ and $\kappa^\mathup{inc} > 1$. Specially, $\kappa_\mathup{eff}^z$ increases in \cref{fig:eff convergence dof a} w.r.t.~$\mathtt{dof}$ while it decreases in \cref{fig:eff convergence dof b}.

\begin{figure}[!ht]
  \centering
  \begin{subfigure}[b]{0.42\textwidth}
    \includegraphics[width=\textwidth]{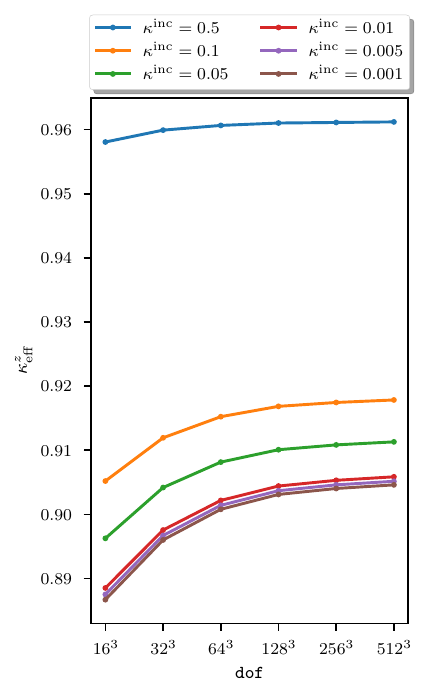}
    \caption{}\label{fig:eff convergence dof a}
  \end{subfigure}
  \begin{subfigure}[b]{0.42\textwidth}
    \includegraphics[width=\textwidth]{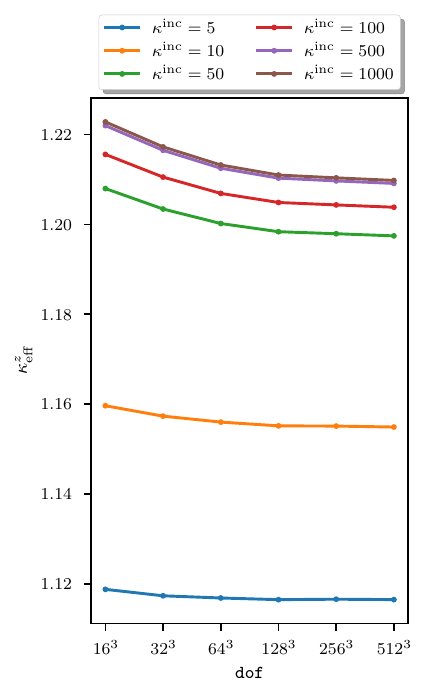}
    \caption{}\label{fig:eff convergence dof b}
  \end{subfigure}
  \caption{Convergences of $\kappa^z_\mathup{eff}$ with  respect to $\mathtt{dof}$ in the center-ball RVE configuration, where $\kappa^\mathup{inc} < 1$ in (a) and $\kappa^\mathup{inc} > 1$ in (b).}\label{fig:eff convergence dof}
\end{figure}

Various convergence tests for FFT-based homogenization schemes are extensively discussed in \cite{Schneider2021}. In this paper, we adhere to the conventional rule and employ the exit criterion given by \cref{eq:exit}. We present the results by selecting $\mathtt{rtol}$ ranging from $\num{1.0e-5}$ to $\num{1.0e-9}$ and different $\kappa^\mathup{inc}$ in \cref{tab:ball center rtol}. It can be concluded that if $\mathtt{rtol}$ is less than $\num{1.0e-5}$, the calculated $\kappa^z_\mathup{eff}$ remains constant, up to four significant digits, across all tested $\kappa^\mathup{inc}$ values. Furthermore, we can observe that the convergence of PCG deteriorates as the \alert{contrast ratio, defined as $\max\{ \kappa^\mathup{inc}, 1/\kappa^\mathup{inc}\}$}, increases. Constructing contrast stable preconditioners usually involves spectral problems and adaptivity techniques \cite{Dolean2015,Ye2024,Ye2024a}. However, those preconditioners in implementations, based on our experience, are complicated for GPU programming and prevail only on some extreme high contrast scenarios ($\geq \num{1.0e+5}$).

\begin{table}[!ht]
  \footnotesize
  \caption{The influence of $\mathtt{rtol}$ and $\kappa^\mathup{inc}$ to $\kappa^z_\mathup{eff}$ and $\mathtt{iter}$ in the center-ball RVE configuration, where $\mathtt{dof}=512^3$.} \label{tab:ball center rtol}
  \centering
  \makegapedcells
  \begin{tabular}{c c c c c c c c c c c c c}
    \hline
    \multirow{2}{*}{$\mathtt{rtol}$} & \multicolumn{2}{c}{$\kappa^\mathup{inc}=0.001$} & \multicolumn{2}{c}{$\kappa^\mathup{inc}=0.01$} & \multicolumn{2}{c}{$\kappa^\mathup{inc}=0.1$} & \multicolumn{2}{c}{$\kappa^\mathup{inc}=10$} & \multicolumn{2}{c}{$\kappa^\mathup{inc}=100$} & \multicolumn{2}{c}{$\kappa^\mathup{inc}=1000$}                                                                                                                                     \\
    \cline{2-13}
                                     & $\kappa_\mathup{eff}^z$                         & $\mathtt{iter}$                                & $\kappa_\mathup{eff}^z$                       & $\mathtt{iter}$                              & $\kappa_\mathup{eff}^z$                       & $\mathtt{iter}$                                & $\kappa_\mathup{eff}^z$ & $\mathtt{iter}$ & $\kappa_\mathup{eff}^z$ & $\mathtt{iter}$ & $\kappa_\mathup{eff}^z$ & $\mathtt{iter}$ \\
    \hline
    \num{1.0e-5}                     & \num{0.905}                                     & \num{9}                                        & \num{0.906}                                   & \num{10}                                     & \num{0.918}                                   & \num{7}                                        & \num{1.155}             & \num{9}         & \num{1.204}             & \num{18}        & \num{1.210}             & \num{43}        \\
    \num{1.0e-6}                     & \num{0.905}                                     & \num{15}                                       & \num{0.906}                                   & \num{12}                                     & \num{0.918}                                   & \num{9}                                        & \num{1.155}             & \num{12}        & \num{1.204}             & \num{24}        & \num{1.210}             & \num{54}        \\
    \num{1.0e-7}                     & \num{0.905}                                     & \num{21}                                       & \num{0.906}                                   & \num{17}                                     & \num{0.918}                                   & \num{11}                                       & \num{1.155}             & \num{14}        & \num{1.204}             & \num{30}        & \num{1.210}             & \num{69}        \\
    \num{1.0e-8}                     & \num{0.905}                                     & \num{31}                                       & \num{0.906}                                   & \num{24}                                     & \num{0.918}                                   & \num{14}                                       & \num{1.155}             & \num{17}        & \num{1.204}             & \num{37}        & \num{1.210}             & \num{85}        \\
    \num{1.0e-9}                     & \num{0.905}                                     & \num{38}                                       & \num{0.906}                                   & \num{29}                                     & \num{0.918}                                   & \num{17}                                       & \num{1.155}             & \num{20}        & \num{1.204}             & \num{44}        & \num{1.210}             & \num{100}       \\
    \hline
  \end{tabular}
\end{table}

\subsection{Stability comparisons with standard preconditioners}
In this subsection, we assess the stability of the proposed preconditioner in comparison to the Incomplete Cholesky (ICC) and Symmetric Successive Over-Relaxation (SSOR) preconditioners \cite{Saad2003}. Both SSOR and ICC preconditioners can be implemented conveniently using routines provided by cuSPARSE. However, we found that for problems with $\mathtt{dof}=512^3$, neither SSOR nor ICC can be executed successfully due to exceeding the available GPU memory (48 GB). Therefore, the relatively small memory usage of the proposed preconditioner is an additional advantage worth noting.

We consider the center-ball RVE configuration introduced in the previous subsection. The convergence histories corresponding to the proposed preconditioner (referred to as FCT), SSOR with $\omega \in \CurlyBrackets{0.5, 1.0, 1.5}$, and ICC for different $\mathtt{dof}$ and $\kappa^\mathup{inc}$ are demonstrated in \cref{fig:convergen history ball center}, where the y-axis in each subplot represents the relative residual, \add{and x-axis represents a logarithmic scale of iteration steps to enhance visual clarity}. \add{We also conduct a test using CG without preconditioning as a reference, denoted by ``\texttt{no-pc}'' in the \cref{fig:convergen history ball center}.} Upon initial inspection, it is evident that the proposed preconditioner offers a significant advantage over the others. Furthermore, the convergence of SSOR and ICC deteriorates as $\mathtt{dof}$ increases, indicating that preconditioning is unable to eliminate the dependence of $\Cond(\mathtt{\mathtt{A}})$ on $\mathtt{dof}$. \add{Moreover, it is worth noting that without preconditioning, CG iterations exhibit slow convergence and fail to meet the convergence criterion within $1024$ iterations in all $\mathtt{dof}=400^3$ cases.}

\begin{figure}[!ht]
  \centering
  \includegraphics[width=\textwidth]{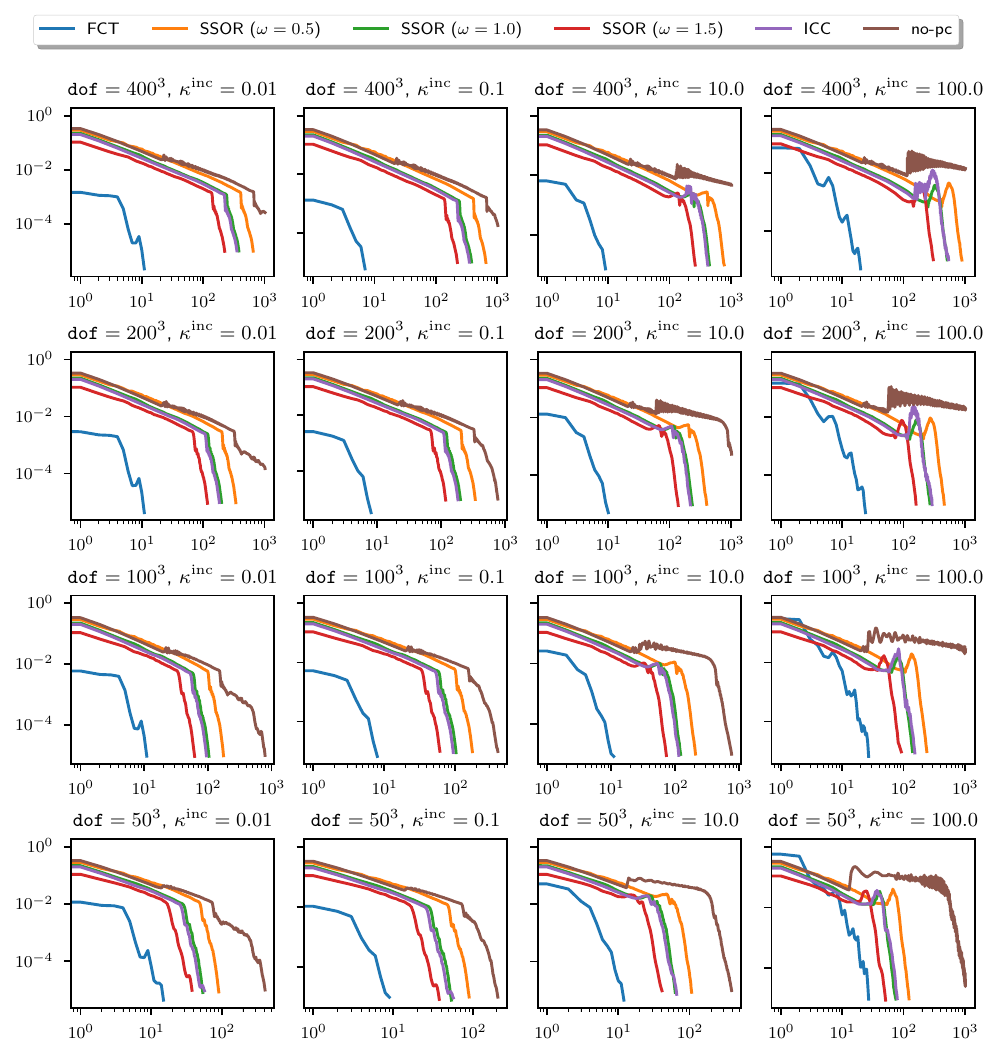}
  \caption{Convergence histories for different preconditioners in the center-ball RVE configuration, \add{where in each plot, the x-axis represents the PCG iteration step and the y-axis represents the relative residual}.}\label{fig:convergen history ball center}
\end{figure}

\begin{figure}[!ht]
  \centering
  \begin{subfigure}[b]{0.32\textwidth}
    \includegraphics[width=\textwidth]{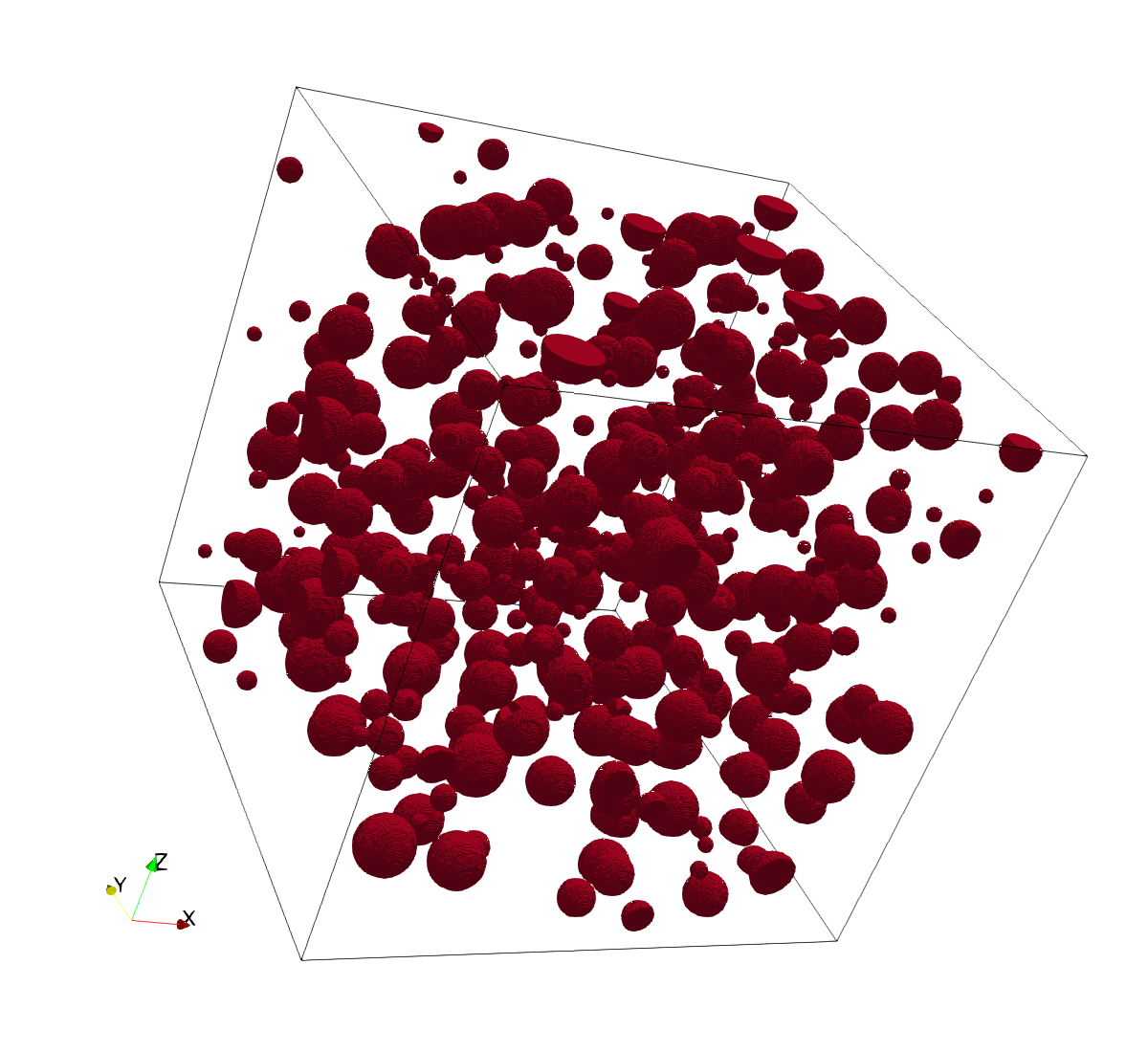}
    \caption{}\label{fig:balls-a}
  \end{subfigure}
  \begin{subfigure}[b]{0.32\textwidth}
    \includegraphics[width=\textwidth]{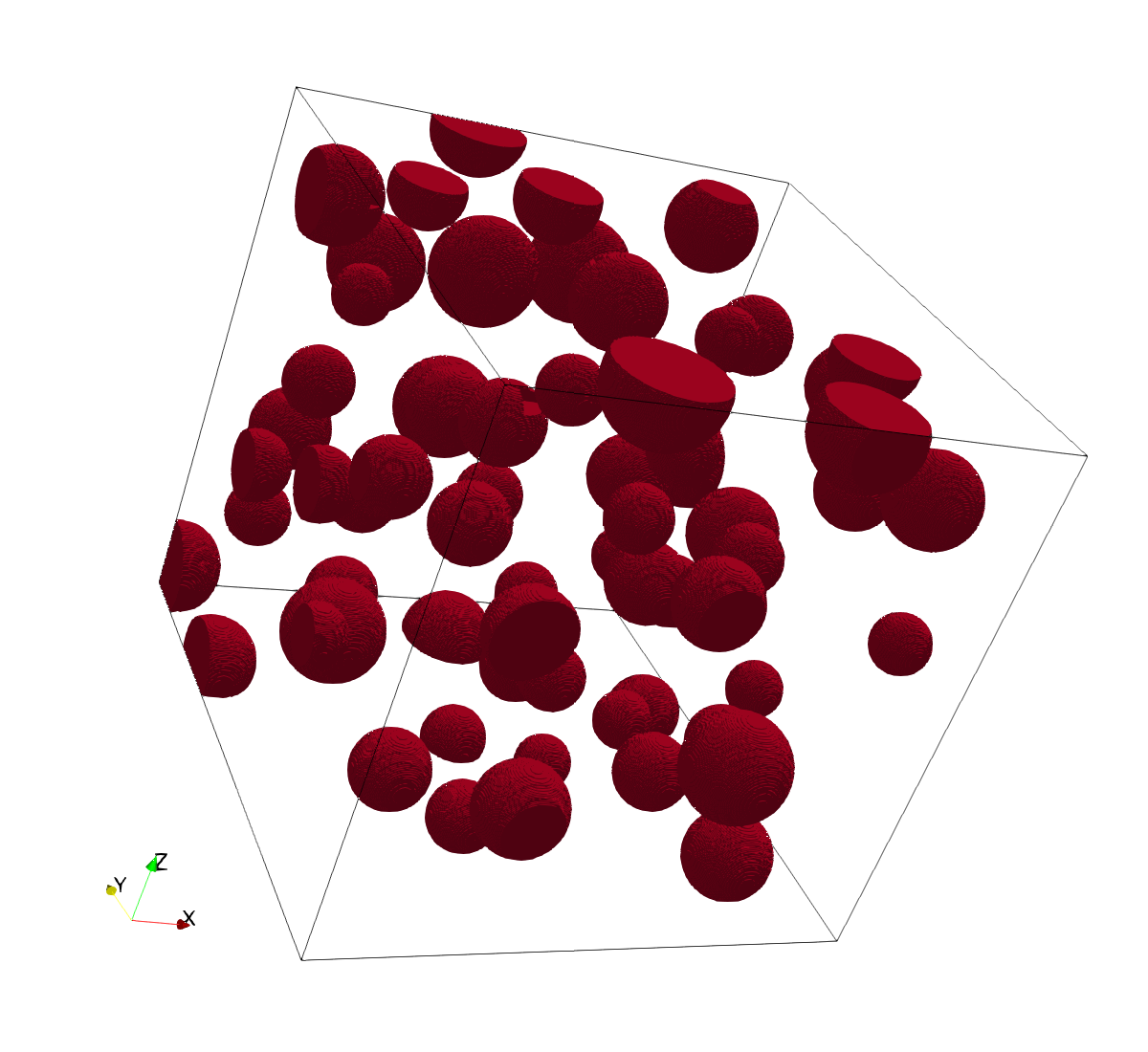}
    \caption{}\label{fig:balls-b}
  \end{subfigure}
  \begin{subfigure}[b]{0.32\textwidth}
    \includegraphics[width=\textwidth]{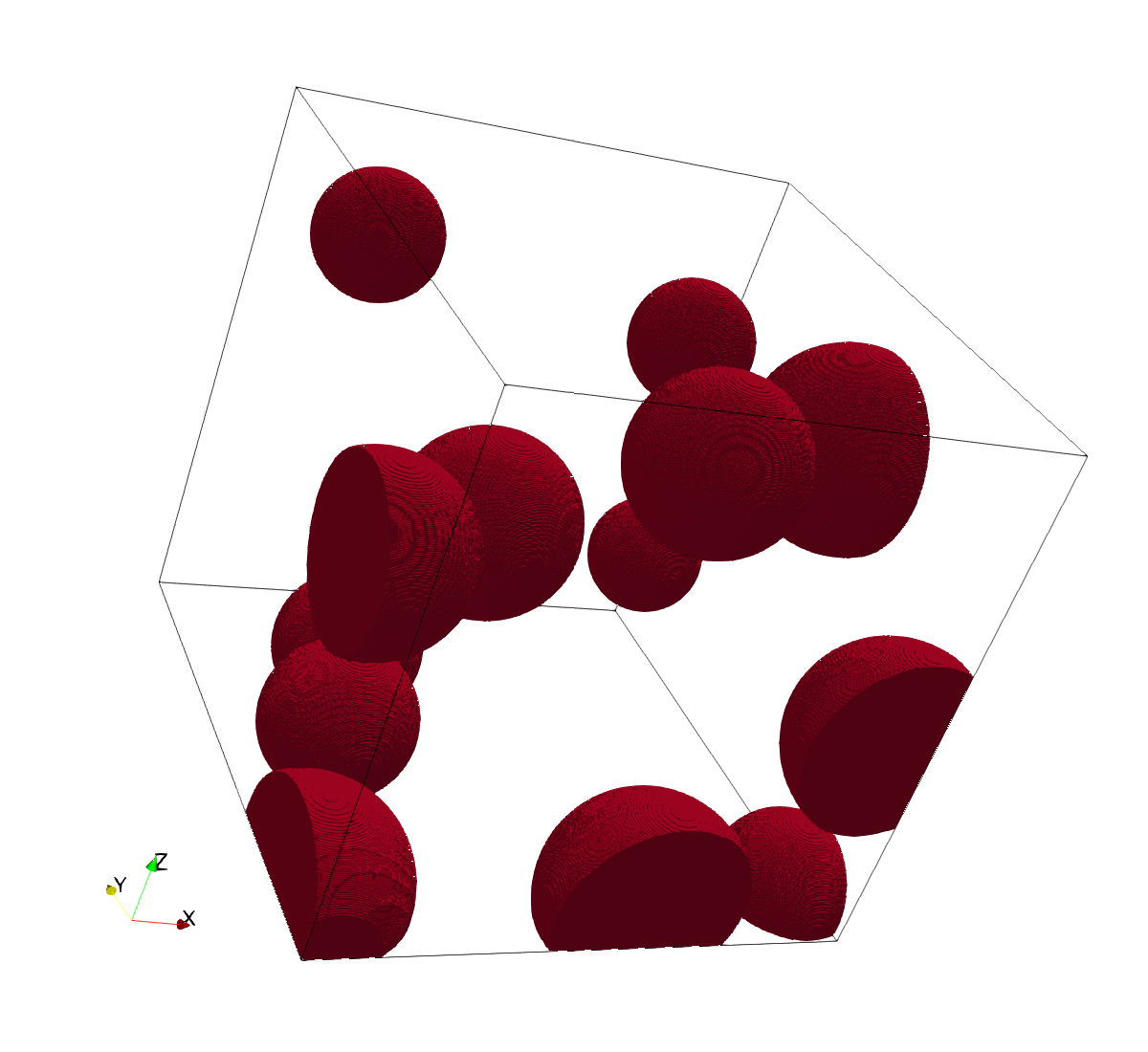}
    \caption{}\label{fig:balls-c}
  \end{subfigure}
  \caption{Three heterogeneous RVE configurations where balls (inclusions) with different sizes are randomly distributed, referred to as config-(a), config-(b), and config-(c) respectively.}\label{fig:ball packed config}
\end{figure}

To provide more compelling examples, we construct and visualize three heterogeneous RVE configurations where balls are randomly distributed, as shown in \cref{fig:ball packed config}. These configurations are referred to as config-(a), config-(b), and config-(c). The discretized coefficient profiles $\kappa^x_{i,j,k}$, $\kappa^y_{i,j,k}$, and $\kappa^z_{i,j,k}$ are obtained using the same procedure as the center-ball configuration, where we set $\mathtt{dof}$ to be either $400^3$ or $512^3$. The convergence histories corresponding to different preconditioners are illustrated in \cref{fig:history balls packed}, where $\mathtt{dof}$ is fixed at $400^3$ due to the limitation that SSOR and ICC cannot be executed within the available GPU memory for the $\mathtt{dof}=512^3$ case. From \cref{fig:history balls packed}, we observe that in some cases, SSOR and ICC fail to reach the convergence criterion within $1000$ iterations. By comparing this with \cref{fig:convergen history ball center} for a simple RVE configuration, we can conclude that the convergence of SSOR and ICC is dependent on the heterogeneity of the coefficients. Meanwhile, the proposed preconditioner shows no stability loss for models with varying levels of heterogeneity. This can be explained by \cref{lem:cond}, which states that the condition number of the preconditioned system can be bounded by the contrast ratios.

\begin{figure}[!ht]
  \centering
  \includegraphics[width=\textwidth]{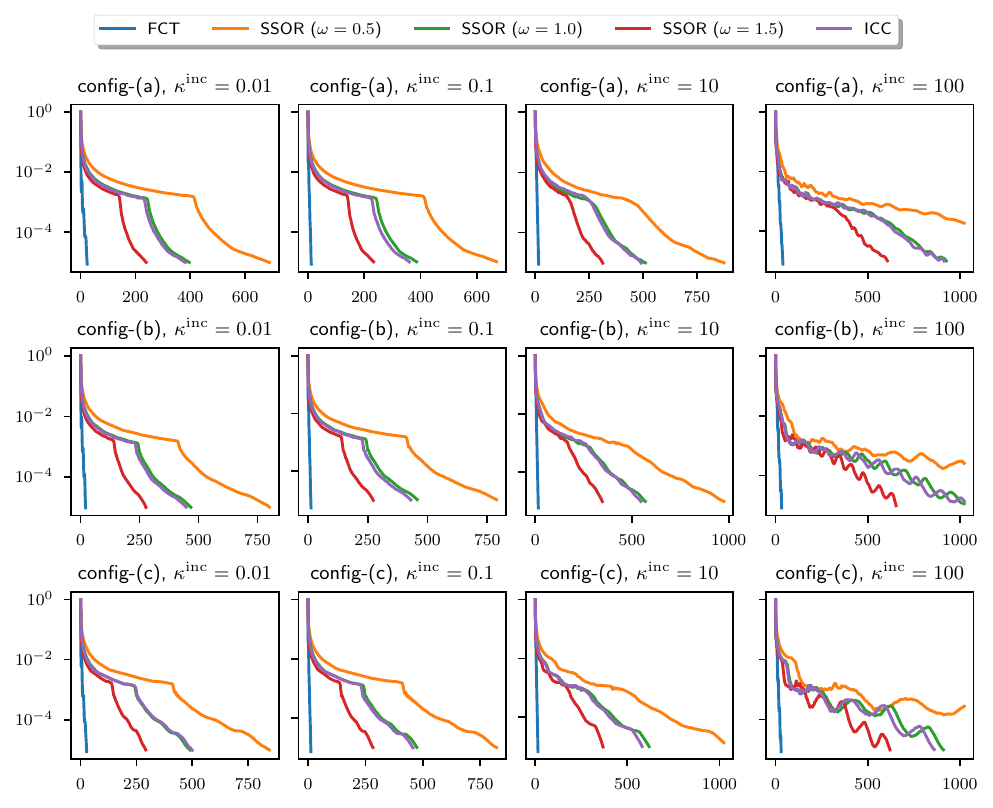}
  \caption{Convergence histories for different preconditioners in the three random RVE configurations, where $\mathtt{dof}$ is fixed as $400^3$, \add{and in each plot, the x-axis represents the PCG iteration step and the y-axis represents the relative residual}.}\label{fig:history balls packed}
\end{figure}

\add{
  Before closing this subsection, we would like to comment on the computational cost of ICC and SSOR in comparison to the proposed preconditioner in one PCG iteration. Once the data for ICC and SSOR are prepared, the computational overhead of ICC and SSOR mainly depends on the sparse triangular system solvers. By exploiting the sparsity of the triangular system, the theoretical cost scales proportionally to the number of DoF, i.e., $\mathcal{O}(N^xN^yN^z)$. However, this analysis does not account for the potential acceleration through parallelism. Due to the intrinsic sequential nature, parallelizing sparse triangular solvers is challenging, making it a prominent topic in the field of parallel computing \cite{Jin2020}. As previously mentioned, our implementation of ICC and SSOR cannot handle the $\mathtt{dof}=512^3$ case due to the GPU memory limitation, we may infer that several sacrifices of space complexity for time complexity are adopted by the solvers behind the scenes \cite{Naumov2011}. On the contrary, FFT algorithms inherently lend themselves to parallelization through the divide-and-conquer strategy, with practical efficiency well-supported by hardware vendors. We also emphasize that tridiagonal systems in the proposed preconditioner are all of size $N^z$ and are independent of each other. While it is possible to merge these small-size tridiagonal systems into a larger one and solve it in parallel, our experience suggests that such a parallelization strategy performs poorly compared to the current implementation, despite the identical theoretical computational cost. This observation also supports our conclusion: although a primary analysis suggests that the proposed preconditioner shares a similar computational cost as ICC and SSOR within a PCG iteration, its practical efficiency is notably superior.
}

\subsection{Performance tests of different implementations}
In this subsection, we evaluate the performance of the two aforementioned platforms, CUDA and MKL-FFTW3, by measuring the wall time. For MKL-FFTW3, parallel acceleration is achieved by using appropriate openMP threads.

Solving the preconditioner system is an essential module in \cref{alg:pcg}, and we first examine the performance of two platforms on a single round of $\mathtt{A}_\mathup{ref}^{-1}\mathtt{r}$. In realistic implementations, it is advisable to introduce two phases, namely preparation and execution, to maximize efficiency. Because the preconditioner system solver needs to be repeatedly executed, operations that are performed only once should be arranged into the preparation phase. Specially, on MKL-FFTW3, the preparation phase encompasses the multiple FCT initialization and tridiagonal matrix factorizations. FFTW3 offers various ``plan'' options to provide flexibility in different situations. The underlying principle of FFTW3 is investing more time in the preparation leads to time savings during FCT/FFT executions. We report the results of using two plans, referred to as ``Plan(P)''  and ``Plan(E)'' \footnote{``Plan(P)'' and ``Plan(E)'' stand for ``\texttt{FFTW\_PATIENT}'' and ``\texttt{FFTW\_ESTIMATE}'' respectively, which are planning-rigor flags used by FFTW3.} in \cref{fig:CPU precond time}, where $\mathtt{dof}$ ranges from $64^3$ to $512^3$. According to the left part of \cref{fig:fftw3 precond solver}, Plan(P) requires significantly more time during the preparation phase. For instance, in the case of $\mathtt{dof}=512^3$, the elapsed time for Plan(P) is almost $\qty{170}{s}$, whereas for Plan(E) it is less than $\qty{2}{s}$. However, reading from the right part of \cref{fig:fftw3 precond solver}, the execution time savings achieved by Plan(P) may not be substantial. For $\mathtt{dof}=512^3$, the average elapsed time for Plan(E) is around $\qty{2.5}{s}$ and for Plan(P) is approximately $\qty{1.5}{s}$. Therefore, it is reasonable to consider Plan(E) in most cases. The same test configuration is also performed on CUDA, and the results are presented in \cref{fig:CPU precond time}. We can observe that the preparation time is almost the same for $\mathtt{dof}=64^3$ to $256^3$, which does not follow the expected linear scaling law. This may be attributed to that the time consumed in transferring data from the main memory to the GPU memory is also included in the preparation phase. The execution time is $\qty{1.8}{ms}$, $\qty{7.0}{ms}$, $\qty{91.0}{ms}$ and $\qty{883.9}{ms}$ respectively, exhibiting a linear, albeit not perfect, increase. We can also observe an acceleration achieved by using GPUs over FFTW3, even when employing an aggressive plan. It is worth noting that FCTs are not out-of-the-box functionalities in the CUDA platform. With the involvement of specialists, we can expect even more significant performance improvements.

\begin{figure}[!ht]
  \centering
  \begin{subfigure}[b]{0.59\textwidth}
    \includegraphics[width=\textwidth]{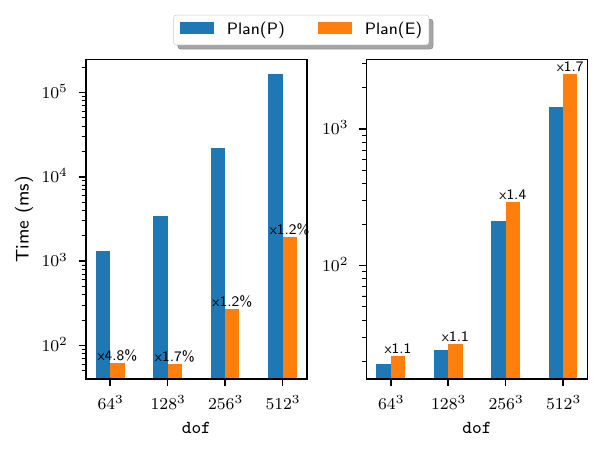}
    \caption{}\label{fig:fftw3 precond solver}
  \end{subfigure}
  \begin{subfigure}[b]{0.4\textwidth}
    \includegraphics[width=\textwidth]{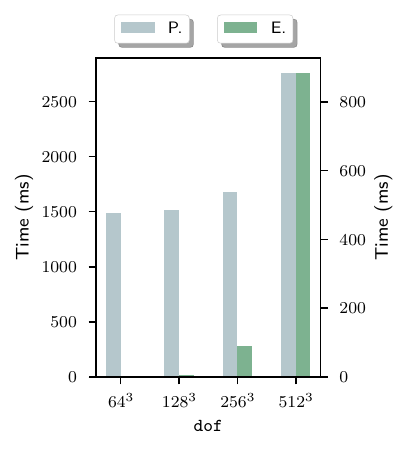}
    \caption{}\label{fig:cuda precond solver}
  \end{subfigure}
  \caption{Preparation and execution time in the preconditioner system solver on MKL-FFTW3 and CUDA: (a) MKL-FFTW3 with \textit{left} preparation and \textit{right} execution; (b) CUDA with ``P.'' for preparation and ``E.'' for execution.}\label{fig:CPU precond time}
\end{figure}

We then proceed to compare MKL-FFTW3 and CUDA in the PCG solver. We target the ball-center configuration, which allows adjusting $\mathtt{dof}$. The total elapsed time to obtain the PCG solution, including the preparation phase, and $\mathtt{iter}$ corresponding to different $\mathtt{dof}$ and $\kappa^\mathup{inc}$ are presented in \cref{fig:cuda cg solver}, with $\mathtt{rtol}$ fixed at $\num{1.0e-5}$ for all cases. We can observe that for all coefficient profiles, $\mathtt{iter}$ remains less than $50$. Although, $\mathtt{iter}$ is slightly larger in the cases of $\kappa^\mathup{inc}=100$ and $1000$.  Moreover, there is a consistent acceleration of the CUDA platform over MKL-FFTW3, as demonstrated in \cref{fig:cuda cg solver}. According to the raw data, in the case of $\mathtt{dof}=512^3$ and $\kappa^\mathup{inc}=1000$, the average elapsed time is around $\qty{108}{s}$ on MKL-FFTW3, while it is approximately $\qty{22}{s}$ on CUDA, showcasing a $5$-fold acceleration in this scenario.

We also examine config-(b) from \cref{fig:ball packed config}, while $\mathtt{dof}$ is now set to $512^3$. The results of the total elapsed time to obtain the PCG solution and $\mathtt{iter}$ regarding different $\kappa^\mathup{inc}$ are illustrated in \cref{fig:cuda-cpu cg solver}. We can confirm that at least a $5$-fold acceleration is achieved across all test cases on the CUDA platform. Even without GPUs, the proposed method can solve models with $\mathtt{dof}=512^3$ in less than $\qty{90}{s}$. We acknowledge that directly comparing wall time records from different methods may be unfair, given the diversities in hardware, platforms, and test configurations. However, based on several available records in \cite{Yang2022,Liu2019}, the proposed method has delivered compelling performance.
Once again, our method relies on four modules that are common facilities across all computing disciplines, making it easier to leverage advancements in new hardware and platforms.

\begin{figure}[!ht]
  \centering
  \begin{subfigure}[b]{0.49\textwidth}
    \includegraphics[width=\textwidth]{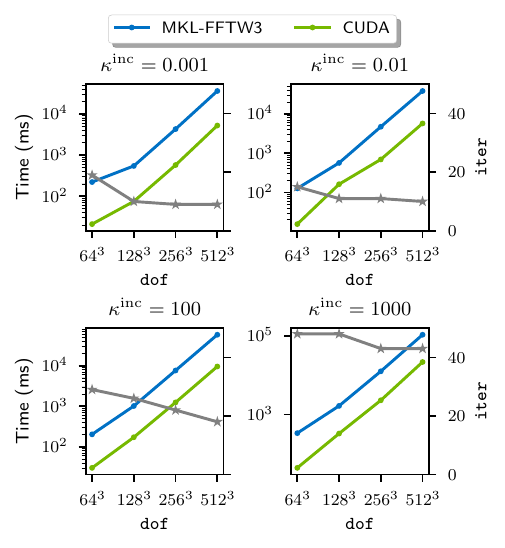}
    \caption{}  \label{fig:cuda cg solver}
  \end{subfigure}
  \begin{subfigure}[b]{0.49\textwidth}
    \includegraphics[width=\textwidth]{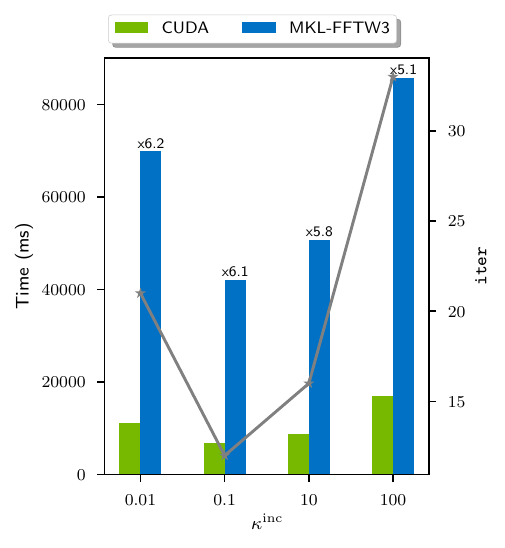}
    \caption{}  \label{fig:cuda-cpu cg solver}
  \end{subfigure}
  \caption{The total elapsed time to obtain the PCG solution on the MKL-FFTW3 and CUDA platforms: (a) the center-ball configuration; (b) config-(b) demonstrated in \cref{fig:ball packed config} with $\mathtt{dof}=512^3$.}
\end{figure}

\subsection{Other experiments}
In this subsection, we offer two numerical experiment designs. The first one aims to validate the effectiveness of deriving reference parameters through linear programming, as stated in \cref{alg:refer parameters}. The second one focuses on investigating the influence and performance of using lower precision functionalities, i.e., changing the data type from the double to the single precision floating point format.

In the first experiment, the coefficient field $\mathbb{K}(\bm{x})$ is constructed from periodically duplicating a cell that is shown in \cref{fig:channels-cell}. Specifically, within the dilated \replace{periodical}{periodic} cell $(0,1)^3$, $\hat{\mathbb{K}}(\bm{y})$ is defined by
\[
  \hat{\mathbb{K}}(\bm{y})=\begin{cases}
    \Diag(2^\Psi,\ 5^\Psi,\ 10^\Psi), & (0,1)\!\!\times\!\!(\frac{3}{8},\frac{5}{8})\!\!\times\!\!(\frac{3}{8},\frac{5}{8}) \cup  (\frac{3}{8},\frac{5}{8})\!\!\times\!\!(0,1)\!\!\times\!\!(\frac{3}{8},\frac{5}{8}) \cup (\frac{3}{8},\frac{5}{8})\!\!\times\!\!(\frac{3}{8},\frac{5}{8})\!\!\times\!\!(0,1), \\
    \Diag(0.01,\ 0.1,\ 1),            & \text{otherwise},
  \end{cases}
\]
where a positive parameter $\Psi$ controls anisotropy. The coefficient field $\mathbb{K}(\bm{x})$ is then determined by $\hat{\mathbb{K}}(\CurlyBrackets{N\bm{x}})$, where $\CurlyBrackets{\cdot}$ represents the fractional part. \Cref{fig:channels-domain} is an illustration of the RVE configuration with $N=4$. It is worth noting that such a long channel configuration with high contrast and strong anisotropy poses challenges for traditional domain decomposition and algebraic multigrid preconditioners \cite{Dolean2015,Ye2024}. For testing PCG solvers, we choose $\mathtt{dof}=512^3$ and $N=64$. To determine the reference parameters, we employ two methods: the first method, denoted as ``one'', assigns all parameters a value of $1$; the second method, referred to as ``opt'', solves the linear programming problem described in \cref{alg:refer parameters}. The convergence histories regarding different $\Psi$  using both ``one'' and ``opt'' are presented in \cref{fig:channels-history}. As expected, the convergence deteriorates as $\Psi$ increases. However, for $\Psi=2$ or $3$, there is a significant improvement with ``opt'' compared to ``one'', while only a slight advancement is observed for $\Psi=1$. Therefore, in situations with strong anisotropy, it is advisable to optimize the reference parameters to achieve faster convergence. In classical FFT-based homogenization methods for elasticity, typically two adjustable parameters can be determined in advance using simple rules \cite{Vinogradov2008}. The proposed \cref{alg:refer parameters} extends these rules to the specific discretization scheme and provides flexibility for tackling challenging anisotropic problems.

\begin{figure}[!ht]
  \centering
  \begin{subfigure}[b]{0.33\textwidth}
    \includegraphics[width=\textwidth]{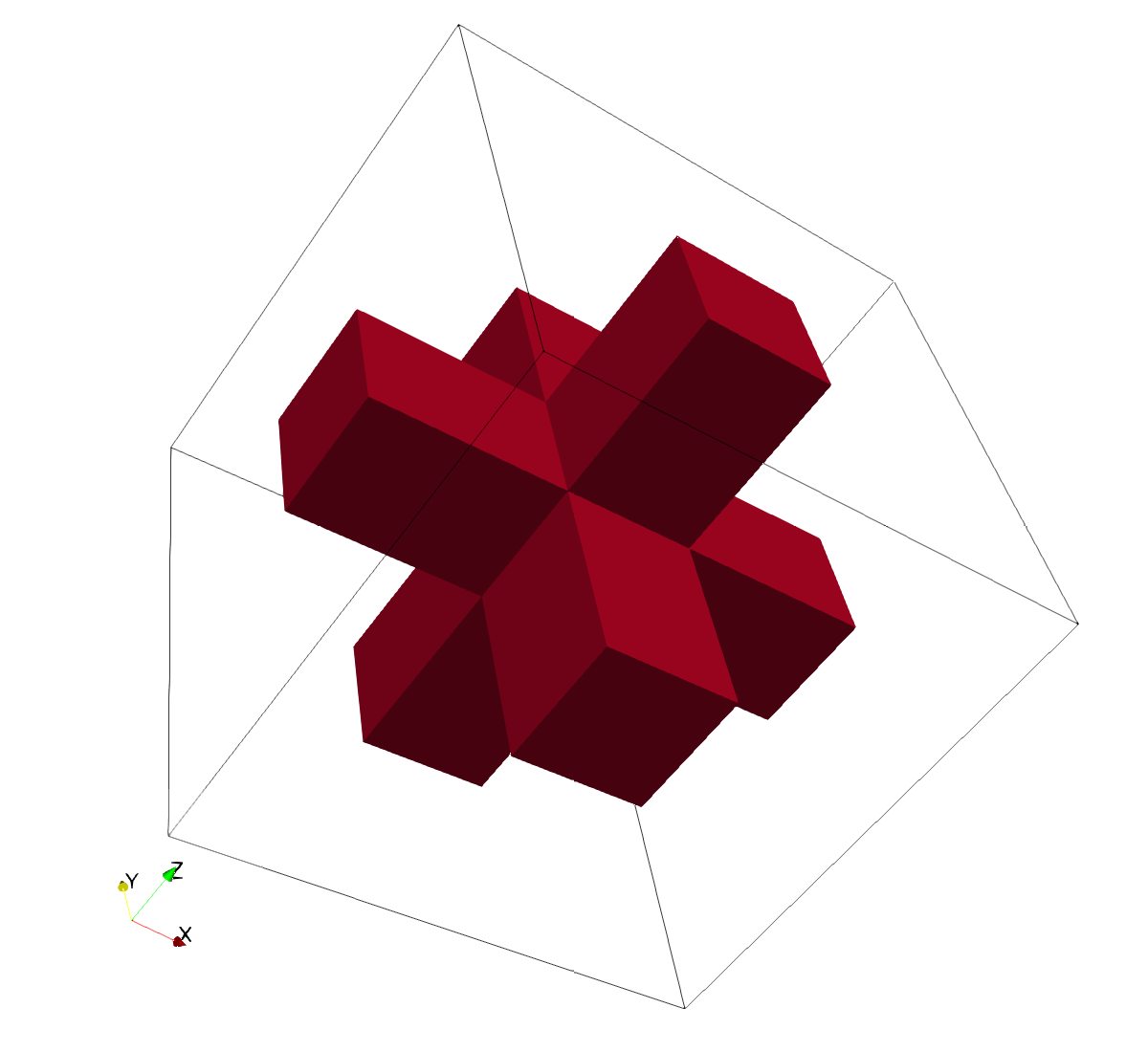}
    \caption{}\label{fig:channels-cell}
  \end{subfigure}
  \begin{subfigure}[b]{0.33\textwidth}
    \includegraphics[width=\textwidth]{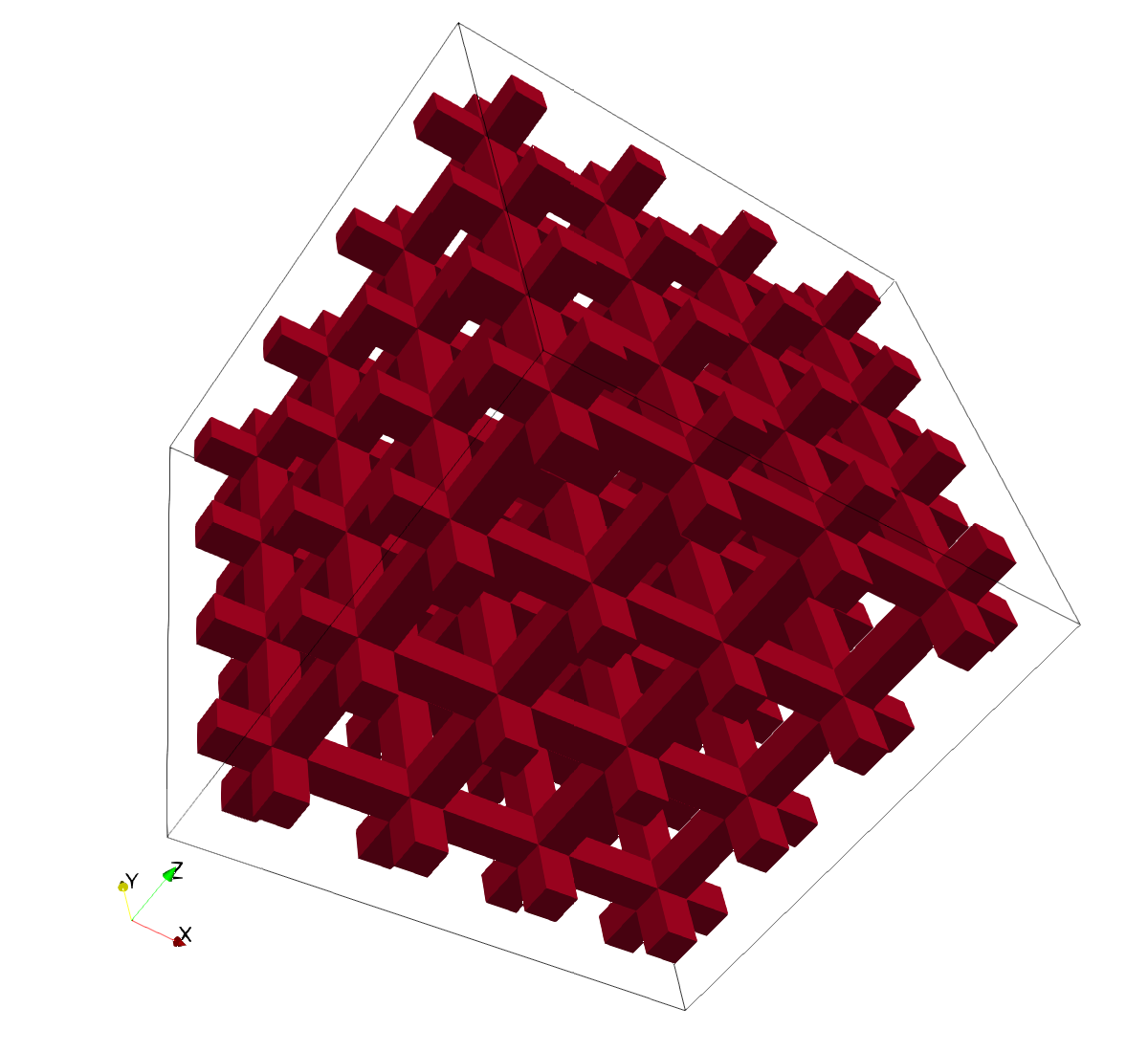}
    \caption{}\label{fig:channels-domain}
  \end{subfigure}
  \begin{subfigure}[b]{0.8\textwidth}
    \includegraphics[width=\textwidth]{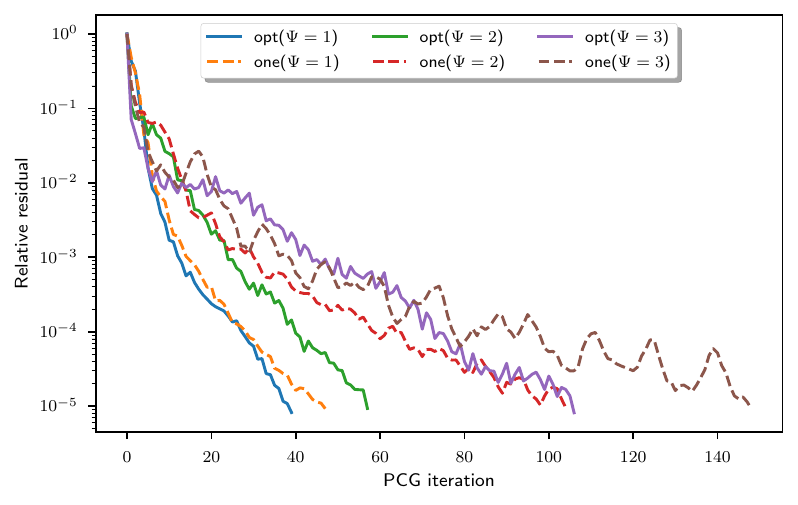}
    \caption{}\label{fig:channels-history}
  \end{subfigure}
  \caption{(a) the \replace{periodical}{periodic} cell, three channels along $x$-, $y$- and $z$-direction are colored in red; (b) a demonstration of an RVE domain composing of $4\times4\times4$ \replace{periodical}{periodic} cells; (c) convergence histories corresponding to different reference parameters under various levels of anisotropicity.}
\end{figure}

The use of single-precision floating points in machine learning tasks is pervasive, primarily due to their advantages in improved memory and computational efficiency, as well as tolerance to noise. Notably, hardware vendors like Nvidia, recognizing the community's preference for single-precision, have invested significant efforts in optimizing the performance of single-precision functions. In homogenization scenarios, where the focus is often on effective coefficients that represent averaged quantities, the choice of precision formats may seem less critical. To validate this proposal, we conduct the second experiment using config-(a) in \cref{fig:ball packed config} to generate coefficient profiles with a resolution of $\mathtt{dof}=512^3$. The results are presented in \cref{tab:single double}, where double precision with a tolerance of $\mathtt{rtol}=\num{1.0e-9}$ is taken as the baseline. We can see that, the differences by employing single precision are all less than $\num{1.0e-2}$, which are generally considered acceptable in engineering applications. It is worth noting that the differences corresponding to high contrast ratios ($\kappa^\mathup{inc}=0.01$ and $\kappa^\mathup{inc}=100$) are larger than those for low contrast ratios ($\kappa^\mathup{inc}=0.1$ and $\kappa^\mathup{inc}=10$). Furthermore, reading from \cref{tab:single double}, we observe that if the single-precision format is adopted, there is an inherent error that cannot be eliminated even with stricter convergence criteria.
One apparent advantage of utilizing single precision is the reduction in memory usage by half, enabling the handling of large RVE models on GPUs, where limited memory is often a major constraint. We are also interested in quantifying the actual computing time savings achieved by single precision. The results are presented in \cref{fig:single vs double}, where the records with a tolerance of $\mathtt{rtol}=\num{1.0e-5}$ using double precision as baselines. It can be observed that under the same tolerance, using single precision can save approximately $40\%$ of the computing time. Balancing accuracy and efficiency hence becomes a major concern, as evident from the comparison with \cref{tab:single double}. Our studies provide an initial understanding of the influence of different floating-point formats on homogenization \cite{Higham2022}, while a rule of thumb should be examined by the specifics of real-world applications.

\begin{table}[!ht]
  \footnotesize
  \caption{Calculated effective coefficients for config-(a) with different $\mathtt{rtol}$, where ``(d)'' and ``(s)'' in the column of $\mathtt{rtol}$ represent using double and single precision respectively, the bold type value represents the baseline and the rest values means relative differences to the baseline in the same column.}\label{tab:single double}
  \centering
  \makegapedcells
  \begin{tabular}{c c c c c c c c c}
    \hline
    \multirow{2}{*}{$\mathtt{rtol}$} & \multicolumn{2}{c}{$\kappa^\mathup{inc}=0.01$} & \multicolumn{2}{c}{$\kappa^\mathup{inc}=0.1$} & \multicolumn{2}{c}{$\kappa^\mathup{inc}=10$} & \multicolumn{2}{c}{$\kappa^\mathup{inc}=100$}                                                                                         \\
    \cline{2-9}
                                     & $\kappa_\mathup{eff}^z$                        & $\mathtt{iter}$                               & $\kappa_\mathup{eff}^z$                      & $\mathtt{iter}$                               & $\kappa_\mathup{eff}^z$ & $\mathtt{iter}$ & $\kappa_\mathup{eff}^z$ & $\mathtt{iter}$ \\
    \hline
    \num{1.0e-9}(d)                  & \textbf{\num{0.91057}}                         & \num{64}                                      & \textbf{\num{0.922601}}                      & \num{25}                                      & \textbf{\num{1.14525}}  & \num{29}        & \textbf{\num{1.19816}}  & \num{78}        \\
    \hline
    \num{1.0e-5}(s)                  & \num{-1.56e-03}                                & \num{24}                                      & \num{1.15e-04}                               & \num{12}                                      & \num{1.00e-04}          & \num{16}        & \num{-5.49e-03}         & \num{37}        \\
    \num{1.0e-6}(s)                  & \num{3.90e-05}                                 & \num{34}                                      & \num{5.40e-05}                               & \num{15}                                      & \num{5.00e-05}          & \num{19}        & \num{-3.08e-03}         & \num{48}        \\
    \num{1.0e-7}(s)                  & \num{1.14e-04}                                 & \num{45}                                      & \num{8.80e-05}                               & \num{19}                                      & \num{7.00e-05}          & \num{23}        & \num{-3.10e-03}         & \num{58}        \\
    \num{1.0e-8}(s)                  & \num{1.14e-04}                                 & \num{55}                                      & \num{8.90e-05}                               & \num{22}                                      & \num{7.00e-05}          & \num{26}        & \num{-3.10e-03}         & \num{71}        \\
    \num{1.0e-9}(s)                  & \num{1.14e-04}                                 & \num{64}                                      & \num{8.90e-05}                               & \num{25}                                      & \num{7.00e-05}          & \num{29}        & \num{-3.10e-03}         & \num{78}        \\
    \hline
    \num{1.0e-5}(d)                  & \num{1.00e-06}                                 & \num{24}                                      & \num{0.00e+00}                               & \num{12}                                      & \num{0.00e+00}          & \num{16}        & \num{1.00e-05}          & \num{37}        \\
    \hline
  \end{tabular}
\end{table}

\begin{figure}[!ht]
  \centering
  \includegraphics[width=\textwidth]{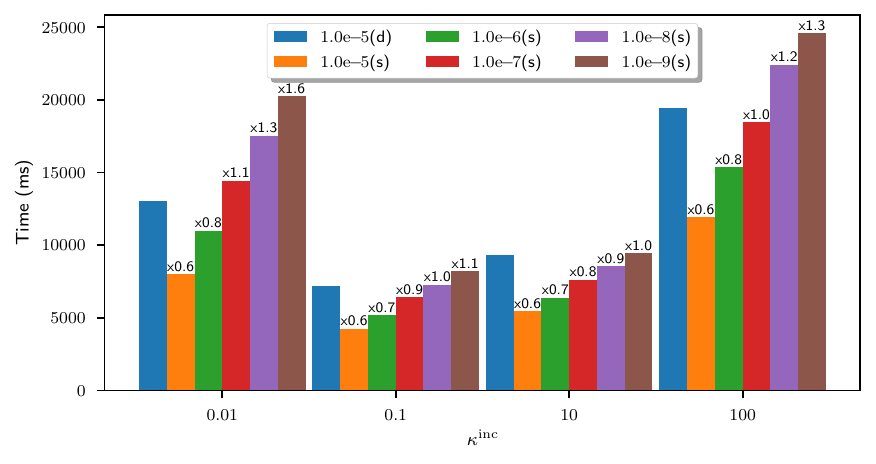}
  \caption{The total elapsed time to obtain the PCG solution by adopting single and double precision for config-(a), where records from $\mathtt{rtol}=\num{1.0e-5}$ with double precision serve as baselines, ``(d)'' and ``(s)'' represent using double and single precision respectively.} \label{fig:single vs double}
\end{figure}

\section{Conclusions}\label{sec:conclusions}
In a post-Moore era, new architectures, technologies, and concepts flood. However, how to fully utilize the surging computing power challenges algorithm designers.  A fundamental consensus should be that if a method mainly relies on common modules provided by specialists who are responsible for the optimized efficiency, the computing performance is more compelling. Nevertheless, such a method may sacrifice generality and become more tailored to specific problems. The proposed method in this paper is a practice of this thinking. We focus on a specific problem\alert{---predicting ETC---}and design an algorithm \replace{closes}{close} to ``vanilla'' functionalities offered by hardware vendors. In particular, we successfully implement the proposed method on the CUDA platform and achieve a $5$-fold acceleration over the pure CPU platform.

The proposed method conceptually consists of two components: discretization and solver. To account for the ETC calculation that the thermal flux should be resolved, we introduce the TPFA scheme to discretize the model PDE. To solve the resulting algebraic linear system, we adopt PCG containing a novel preconditioner to accelerate the convergence. The construction of the preconditioner originates from FFT-based homogenization, which replaces heterogeneous coefficient profiles with a set of homogeneous reference parameters. To enhance the performance of anisotropic models, we design a linear programming technique to determine the reference parameters, which is validated to be effective through numerical experiments. We observe that the preconditioner system could be solved efficiently using multiple FCTs and parallel tridiagonal matrix solvers. A theoretical by-product that provides a two-side estimate for the condition number of the original algebraic linear system is obtained by this observation. Considering that multiple FCTs are not out-of-the-box functionalities on some platforms, we also detail a memory-optimal implementation of FCTs from FFTs. To showcase the efficiency of our proposed method while considering GPU memory limitations, we conduct numerical experiments for 3D RVEs with DoF up to $512^3$. The results indicate that our method remains stable w.r.t.~the number of DoF and heterogeneity of RVEs, which is generally unachievable for standard preconditioners. We address that our method cannot eliminate performance deterioration caused by high contrast ratios. Although there are preconditioners specifically developed to address this issue and claim to be contrast-stable, our method still prevails in terms of its simple implementation and high efficiency.

Nowadays, machine learning-related methods have exhibited a dominant superiority in image processing tasks. Based on pixel/voxel representations of RVEs, it seems to be natural to apply machine learning to homogenization, especially \replace{especially dealing}{when dealing with} nonlinear material laws. The success of machine learning cannot leave the great computing capability powered by GPUs.
\alert{Therefore, our emphasis on the CUDA implementation of our method aims to facilitate its seamless integration with AI.}
Meanwhile, several topics that are important to data-driven frameworks, for example, the influence of floating point precision on homogenization discussed primarily in the article, may attract attention in the future.

\section*{Acknowledgments}
CY owings his special thanks to Dr.~Kwai Lam Wong for his brilliant short course ``From Parallel Computing to Deep Learning'' and his preaching ``GPU is the future'', which \replace{motive}{motived} CY to shape the paper at hand.
SF's research is supported by startup funding from the Eastern Institute of Technology, Ningbo, and NSFC (Project number: 12301514).
EC's research is partially supported by the Hong Kong RGC General Research Fund (Project numbers: 14305222 and 14304021).
\add{They also wish to thank the anonymous referees for their thoughtful comments, which helped in improvement of the presentation.}

\bibliographystyle{siamplain}
\bibliography{refs}
\end{document}